\documentclass{article}
\usepackage{amsmath,subfigure,amssymb,amsthm,graphicx,epsfig,float,url}
\usepackage[colorlinks=true]{hyperref}
\usepackage{tikz}
\usepackage{pgf,tikz}
\usepackage{graphicx}
\usepackage{indentfirst}
\usepackage{pgfplots}
\usepackage{caption} 

\usepackage[applemac]{inputenc}
\usepackage[usenames,dvipsnames]{pstricks}

\newcommand{\BV}{\mathbf{BV}}

\newcommand{\W}[1]{\mathbf{W^{#1}}}

\newcommand{\R}{\mathrm{I\kern-0.21emR}}
\newcommand{\N}{\mathrm{I\kern-0.21emN}}

\newcommand{\C}{\mathbb{C}}

\renewcommand{\geq}{\geqslant}
\renewcommand{\leq}{\leqslant}

\newtheorem{theorem}{Theorem}

\newtheorem{definition}{Definition}
\newtheorem{lemma}{Lemma}
\newtheorem{example}{Example}
\theoremstyle{definition}\newtheorem{remark}{Remark}

\newcommand{\cqfd}
{%
\mbox{}%
\nolinebreak%
\hfill%
\rule{2mm}{2mm}%
\medbreak%
\par%
}

\title{Well-posedness for scalar conservation laws with moving flux constraints}

\author{
Thibault Liard\footnote{Inria Grenoble Rh\^one-Alpes, France (\texttt{thibault.liard@inria.fr}).}
\and
Benedetto Piccoli\footnote{Department of Mathematical Sciences and CCIB, Rutgers University-Camden, Camden,
NJ, USA (\texttt{piccoli@camden.rutgers.edu}).}
}

\date{}
\begin{document}
\maketitle

\begin{abstract}
We consider a strongly coupled ODE-PDE system representing moving bottlenecks immersed in vehicular traffic. 
The PDE  consists of a scalar conservation law  modeling  the traffic flow evolution and the ODE models the trajectory of a slow moving vehicle. The moving bottleneck influences the  bulk traffic flow 
via a point flux constraint, which is given by an inequality on the flux at the slow vehicle position. 
We prove uniqueness and continuous dependence of solutions with respect to initial data of bounded variation. The proof is based on a new backward in time method established to capture the values of the norm of generalized tangent vectors at every time. 
\end{abstract}


\medskip



\section{Introduction and main results}\label{sec:intro}
\subsection{Presentation of the problem}
Macroscopic models, in particular fluid-dynamic ones, for vehicular traffic were extensively
studied and used in recent years in the applied math and engineering communities .
The mains reasons for this success include the many analytic tools available 
\cite{BFLPZ13,BBNS14,BCGHP14,GP06} and
their usability with sensors data (both fixed and probe) \cite{CDP10,GP16,WBTPB10}.
Probe sensors has been successfully implemented for traffic monitoring since mid 2000s \cite{ITS08}
and the new frontiers are in the area of traffic control. A lot of attention is now focused on
Connected and Autonomous Vehicles (CAVs) seen as distributed probe actuators \cite{Arizona17}.
CAVs can be represented as moving bottlenecks and some modeling approach are available,
\cite{DP14,DP17,LMP2011}, based on flux constraints and coupled ODE-PDE systems. 
To develop a complete theoretical framework for traffic control via moving bottleneck,
the main theoretical question to be addressed is the well-posedness of the ODE-PDE systems.
This paper addresses this question for the model introduced by Delle Monache and Goatin in \cite{DP14}.

Let us describe in more detail the ODE-PDE models for moving bottlenecks.
In \cite{LMP2011}, to represent the capacity drop of car flow due to the presence of a slow vehicle, the authors multiply the usual flux function by a cut-off function. To obtain a unique solution in the sense of Fillipov  (\cite{Fil88}), they assume that the slow vehicle travels at maximal speed. In \cite{DP14}, the authors represent the moving constraint by an pointwise inequality on the flux and prove that the Cauchy problem \eqref{eq} admits a solution 
using wave-front tracking approximations. In \cite{DP17},  a proof of the stability of  solutions for a weakly coupled PDE-ODE system is given. The term ``weakly coupled'' means that the position of the slow vehicle is assumed to be assinged. 
Some numerical methods 
have been developed in \cite{DJ05, DJ05b, DP14b}. In \cite{GC17}, the authors 
replace the single conservation law, called Lighthill-Whitham-Richards (briefly LWR) 
first order model \cite{LW55,Ric56}, with a system of conservation laws,
called the Aw-Rascle-Zhang (briefly ARZ) 
second order model \cite{AR00,Zha02}. 
They define two different Riemann Solvers and they propose numerical methods.  

Here we focus on the model proposed in \cite{DP14}, thus
we study the following strongly coupled ODE-PDE system
\begin{equation} \label{eq} \begin{array}{ll}
\partial_t \rho + \partial_x (\rho(1-\rho))=0, & (t,x)\in \R^+\times \R,\\ 
\rho(0,x)=\rho_{0}(x), & x\in \R,\\
f(\rho(t,y(t)))-\dot y(t)\rho(t,y(t)) \leq F_{\alpha}:=\frac{\alpha}{4}(1-\dot y(t))^2, & t\in \R^+,\\ 
\dot y(t)=\omega (\rho(t,y(t)+))), & t\in \R^+,\\ 
y(0)=y_0, & x\in \R.\\
\end{array}
\end{equation}
Above, $\rho=\rho(t,x)\in [0,1]$ is the mean traffic density, $f$ is the flux defined by 
$$ f(\rho)=\rho v(\rho) \quad \text{with} \quad  v(\rho)=1-\rho. $$
The variable $y$ denotes the slow vehicle (briefly SV) position and the velocity of the SV 
is described by :
\begin{equation} \label{velocity}
\omega(\rho)=\left\{ \begin{array}{ll}
V_b &\text{if} \,  \rho \leq \rho^*:=1-V_b,\\
v(\rho) & \text{otherwise,}\\
\end{array} \right. \end{equation}
where $V_b\in (0,1)$ denotes the maximal speed of the SV. 
For future use, we also defined $ \check{\rho}_{\alpha}$ and $ \hat{\rho}_{\alpha}$ with $ \check{\rho}_{\alpha} \leq \hat{\rho}_{\alpha}$ to be the solutions of 
$\frac{\alpha}{4}(1-V_b)^2+V_b \rho=f(\rho)$ and $\rho^*$
to be the solution of $V_b\rho=f(\rho)$.
See also Figure \ref{fig:1}.

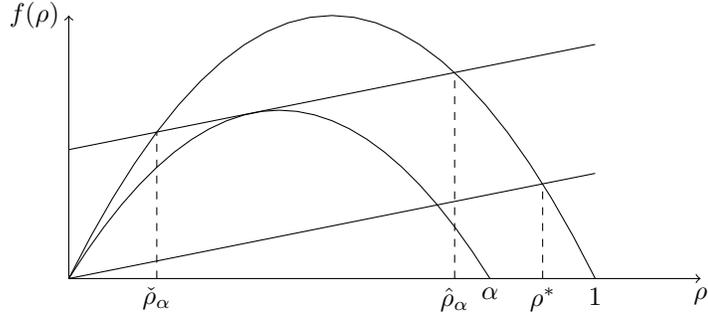
\begin{figure} 
\centering
\begin{tikzpicture}[scale=7]
\draw[->] (0,0) -- (1.2,0);
\draw (1.2,0) node[below] {$\rho$};
\draw [->] (0,0) -- (0,1/2);
\draw [domain=0:1] plot (\x, 2*\x*(1-\x);
\draw [domain=0:0.8] plot (\x, 2*\x*(0.8-\x);
\draw [domain=0:1] plot (\x, 0.2*\x-0.2*0.35+0.315);
\draw [domain=0:1] plot (\x, 0.2*\x);
\draw[dashed] (0.1671572875,0)--(0.1671572875,0.2784314574);
\draw (0.1671572875,0) node[below] {$\check{\rho}_{\alpha}$};
\draw[dashed] (0.7328427125,0)--(0.7328427125,0.3915685424);
\draw (0.7328427125,0) node[below] {$\hat{\rho}_{\alpha}$};
\draw[dashed] (0.9,0)--(0.9,0.18);
\draw (0.9,0) node[below] {$\rho^*$};
\draw (1,0) node[below] {$1$};
\draw (0.8,0) node[below] {$\alpha$};
\draw (0,1/2) node[left] {$f(\rho)$};
\end{tikzpicture}
\caption{Graphical representation of the flux function and  of $ \check{\rho}_{\alpha}$, $ \hat{\rho}_{\alpha}$ and $\rho^*$}\label{fig:1}
\end{figure}

\subsection{Main result}
Let us state the main result of this article. The following  theorem is devoted to 
uniqueness and continuous dependence of solutions for \eqref{eq} with respect to the initial data.

\begin{theorem} \label{main1}
The solution $(\rho,y) \in C^0(\R^+;L^1(\R) \cap BV(\R,[0,1])) \times W^{1,1}(\R^+,\R)$ in the sense of Definition \ref{def:weak-solution-1} for the Cauchy problem  \eqref{eq} depends in a Lipschitz continuous way from the initial datum. More precisely, let $T>0$ and $(\rho^1,y^1)$ and $(\rho^2,y^2)$ two solutions of \eqref{eq} with corresponding initial data $(\rho_0^1,y_0^1)$ and $(\rho_0^2,y_0^2)$, then there exists $C>0$ such that 
$$ \Vert \rho^2(t)-\rho^1(t) \Vert_{L^1(\R)} + \vert y^2(t)-y^1(t) \vert \leq C(  \Vert \rho_0^2-\rho_0^1 \Vert_{L^1(\R)} + \vert y_0^2-y_0^1 \vert),$$
for every $t\in [0,T]$.
\end{theorem}
The article is organized as follows. In Section \ref{WW0} and Section \ref{WW1}, we recall some properties of system \eqref{eq} (the Riemann solver and the existence of solutions). In section \ref{tangent}, we use the notion of generalized tangent vectors to estimate the $L^1$ distance of two different piecewise constant  approximate solutions constructed by wave-front tracking method.  In Section \ref{110}, we introduce a new mathematical object which traces the discontinuities of waves backwards in time. Section \ref{111} deals with all the possible interactions   between two waves and between a wave and the slow vehicle trajectory giving the evolution of the tangent vectors for each interaction. Section \ref{jhjh} is devoted to the proof of Theorem \ref{main1}; from the final state, we manage to follow backwards in time all the discontinuities capturing the evolution of generalized tangent vectors.


\section{Notations and Preliminary materials} \label{WW}
\subsection{The Riemann problem with moving constraints} \label{WW0}
This section is devoted to the study of the Riemann problem. We consider \eqref{eq} with Riemann type initial data
\begin{equation} \label{initial} \rho_0(x)=\left\{ \begin{array}{ll}
\rho_L&\text{if}\quad x<0\\
\rho_R&\text{if}\quad x>0
\end{array} \right. \quad \text{and} \quad y_0=0.
\end{equation}
The definition of the Riemann solver for \eqref{eq} and \eqref{initial} is described in \cite[Section 3]{DP14}; We denote by $\mathcal{R} $ the standard Riemann solver for 

\begin{equation} \label{bobo} \left\{\begin{array}{ll}
\partial_t \rho + \partial_x (\rho(1-\rho))=0, & (t,x)\in \R^+\times \R,\\ 
\rho(0,x)=\rho_{0}(x), & x\in \R,\\
\end{array}\right.
\end{equation}
where $\rho_0$ is defined in \eqref{initial}.
\begin{definition}{ \cite[Section 3]{DP14}}The constrainted Riemann solver $\mathcal{R}^{\alpha}:[0,1]^2 \mapsto \textbf{L}_{\text{loc}}^1(\R;[0,1])$ for \eqref{eq} and \eqref{initial} is defined as follows.
\begin{enumerate}
\item If $f(\mathcal{R}(\rho_L,\rho_R)(V_b))>F_{\alpha}+V_b \mathcal{R}(\rho_L,\rho_R)(V_b)$, then
 \begin{equation*} \mathcal{R}^{\alpha}(\rho_L,\rho_R)(x \slash t)=\left\{ \begin{array}{ll}
\mathcal{R}(\rho_L,\hat \rho_{\alpha})(x \slash t)&\text{if}\quad x<V_b t,\\
\mathcal{R}(\check \rho_{\alpha},\rho_R)(x \slash t)&\text{if}\quad x\geq V_b t,
\end{array} \right. \quad \text{and} \quad y(t)=V_b t.
\end{equation*}
\item If $V_b\mathcal{R}(\rho_L,\rho_R)(V_b) \leq f(\mathcal{R}(\rho_L,\rho_R)(V_b))\leq F_{\alpha}+V_b \mathcal{R}(\rho_L,\rho_R)(V_b)$, then
 \begin{equation*} \mathcal{R}^{\alpha}(\rho_L,\rho_R)=\mathcal{R}(\rho_L,\rho_R) \quad \text{and} \quad y(t)=V_b t.
\end{equation*}
\item If $ f(\mathcal{R}(\rho_L,\rho_R)(V_b))<V_b\mathcal{R}(\rho_L,\rho_R)(V_b) $, then
 \begin{equation*} \mathcal{R}^{\alpha}(\rho_L,\rho_R)=\mathcal{R}(\rho_L,\rho_R) \quad \text{and} \quad y(t)=v(\rho_R) t.
\end{equation*}
\end{enumerate}
\end{definition}
\subsection{The Cauchy problem: existence of solutions} \label{WW1}
We introduce the definition of solutions to the constrained Cauchy problem~(\ref{eq}) as in \cite[Section 4]{DP14}.
\begin{definition}\cite[Section 4]{DP14} 
  \label{def:weak-solution-1}
  The couple
  \begin{displaymath}
    \left( \rho,y \right)  \in \C_0 \left( [0, +\infty[; \textbf{L}^1\cap \BV(\R;[0,1]) \right)\times \W{1,1} \left( [0, +\infty[; \R \right) 
  \end{displaymath}
  is a solution to~(\ref{eq}) if
  \begin{enumerate}
  \item the function $\rho$ is a weak solution to the PDE in~(\ref{eq}), for $(t,x) \in (0,+\infty) \times \R$;

   \item $\rho(0, x) = \rho_0(x)$, for a.e. $x \in \R$;

\item  the function $y$ is a Caratheodory solution to the ODE in~(\ref{eq}), i.e. for a.e. $t \in \R^{+}$ 
\begin{equation}
 y(t)= y_{0}+ \int_{0}^{t}\omega\left( \rho(s,y(s)+) \right)\,ds\,;
\end{equation}
\item  the constraint is satisfied, in the sense that for a.e. $t \in \R^{+}$
\begin{equation}
\lim_{x \rightarrow y(t)\pm}\, \left(f(\rho)-w(\rho) \rho \right)(t,x) \leq F_{\alpha} \,;
\end{equation}
  \end{enumerate}
\end{definition} 
Let $\rho_0 \in BV(\R,[0,1])$. The existence of solutions for \eqref{eq} in the sense of Definition \ref{def:weak-solution-1} is proved in \cite{DP14}. The authors construct a sequence of approximation solutions via the wave-front tracking method and prove its convergence. 

\subsection{ Wave-front tracking and generalized tangent vectors} \label{tangent}
Solutions to Cauchy problems for conservation laws can be constructed by various methods
including wave-front tracking, see \cite{B00,HR15}.
In simple words, wave-front tracking works in the following way.
One first approximate the initial data by piecewise constant
functions, then solve the corresponding Riemann problems and piece solutions together
approximating rarefaction waves with fans of rarefaction shocks. 
Then each wave moves with the speed prescribed by the Rankine-Hugoniot condition and
when two waves meet a new Riemann problem is solved.
Since our problem is scalar, we can use the very first algorithm proposed by Dafermos \cite{Daf72}.
For the system case and application to traffic see \cite{GP06}.

We introduce on $[0,1]$ the mesh $\mathcal{M}_n=\{\rho_i^n\}_{i=0}^{2^n}$ defined by 
 $$\mathcal{M}_n=(2^{-n} \N \cap [0,1]).$$
 To introduce the points $\check{\rho}_{\alpha}, \hat{\rho}_{\alpha}$ and $\rho^*$, we modify the mesh $\mathcal{M}_n$ as in \cite[Section 4.1]{DP14},
 \begin{itemize}
\item if $\min_i |\check{\rho}_{\alpha}-\rho_i^n|=2^{-n-1}$ then we add the point $\check{\rho}_{\alpha} $ to the mesh 
$$\widetilde{\mathcal{M}_n}=\mathcal{M}_n \cup \{ \check{\rho}_{\alpha} \};$$
\item if $|\check{\rho}_{\alpha}-\rho_l^n|=\min_i |\check{\rho}_{\alpha}-\rho_i^n|<2^{-n-1}$ then we replace $\rho_l^n$ by $\check{\rho}_{\alpha} $
$$\widetilde{\mathcal{M}_n}=\mathcal{M}_n \cup \{ \check{\rho}_{\alpha} \} \backslash \{\rho_l^n\};$$
\item we perform the same operation for $\hat{\rho}_{\alpha}$ and for $\rho^*$.
 \end{itemize}
We notice that if $\tilde{\rho}_j^n, \tilde{\rho}_i^n \in \widetilde{\mathcal{M}_n}$ then $\frac{1}{2^{n+1}}\leq \vert \tilde{\rho}_j^n-\tilde{\rho}_i^n\vert \leq \frac{3}{2^{n+1}}$.
For $i=\{1,2\}$, we construct a piecewise constant approximate solution of $\rho_0^i$  denoted by $\rho_0^{i,n}$ such that,  
$$\rho_{0}^{i,n}=\sum_{j=0}^N\rho_{0,j}^{i,n}\chi_{(x^0_{j-1},x^0_j]} \, \, \text{with} \, \, \rho_{0,j}^{i,n} \in \widetilde{\mathcal{M}_n} \quad \text{and} \quad TV(\rho_0^{i,n})\leq TV(\rho_0^{i})$$
which approximates $\rho_0^{i}$ in the sense of the strong $L^1$ topology, that is to say, 
$$ \lim_{n\to \infty}\Vert \rho_0^{i,n}-\rho_0^{i}\Vert_{L^1(\R)}=0.$$
Above $x_1^0<\cdots<x_N^0$ are the points where $\rho_{0}^{i,n}$ is discontinuous. Solving all the Riemann problem for \eqref{bobo} generated by the jump $ (\rho_{0}^{i,n}(x_i^0-),\rho_{0}^{i,n}(x_i^0+))$ for $i=1,\cdots,N$, the solution, denoted by $\rho^{i,n}$, can be prolonged until a first time $t_1$ is reached, when two wave-fronts interact. In the wave-front tracking method, the centered rarefaction waves are  approximated by piecewise constant rarefaction fans where each rarefaction front has strengh less than $\frac{3}{2^{n+1}}$. Thus,  $\rho^{i,n}(t_1,\cdot)$ is still a piecewise constant function, the corresponding Riemann problems can again be approximately solved within the class of piecewise constant functions and so on. Let $y^{i,n}$ the solution of $$ \label{eq3} \left\{ \begin{array}{ll}
\dot y(t)=\omega (\rho^{i,n}(t,y(t)+))), & t\in \R^+,\\ 
y(0)=y_0, & x\in \R.\\
\end{array} \right.
$$
where $\rho^{i,n}(t)$ corresponds to the wave-front tracking approximate solution at time $t$ as described below with initial data $\rho_0^{i,n}$ (see \cite[Section 2.6]{GP06}). We will prove that  
\begin{equation} \label{popo} \Vert \rho^{2,n}(t)-\rho^{1,n}(t) \Vert_{L^1(\R)} + \vert y^{2,n}(t)-y^{1,n}(t) \vert \leq C(T)(  \Vert \rho_0^{2,n}-\rho_0^{1,n} \Vert_{L^1(\R)} + \vert y_0^{2}-y_0^{1} \vert). \end{equation}
We use the notion of generelized tangent vectors , introduced in \cite{B93,BCP00} for systems of conservation laws and adapted to traffic applications in \cite{BP08,GP09}.  The main idea is to estimate the $L^1$-distance viewing $L^1$ as a Riemannian manifold. Let $[a,b] \subset \R$ and $PC$ denotes the set of piecewise constant functions with finitely many jumps. An \textit{elementary path} is a map $\gamma : [a,b] \to PC$ of the form 
$$\gamma(\theta)=(\sum_{j=1}^N \rho_j \chi_{[x_{j-1}^\theta,x_{j}^\theta]},y^{\theta}),$$
where $x_j^\theta=x_j+\xi_{j }\theta$, $y^\theta=y+\xi_b \theta$ with  $x_{j-1}^\theta<x_{j}^\theta$ for every $\theta \in [a,b]$ and $j=1 \cdots N$. 
The length of an elementary path is defined as:
\[
\Vert  \gamma\Vert=\int_a^b \sum_{j} \vert \Delta \rho_j \xi_j\vert + \vert \xi_b \vert \, d\theta,  
\]
and it is easy to check that this is compatible with the usual $L^1$ metric, i.e.
$\Vert  \gamma\Vert=\Vert  \gamma\Vert_{L^1(\R)}$.\\
The functions $(\rho_0^{1,n},y_0^{1})$ and $(\rho_0^{2,n},y_0^{2})$ can be joined by a piecewise elementary path $\gamma_0$ with a finite number of pieces. 
If we denote by $\gamma_t(\theta)$  the path
obtained at time $t$ via wave-front tracking, then: 
\begin{equation} \label{test1}
\Vert \rho^{2,n}(t)-\rho^{1,n}(t) \Vert_{L^1(\R)} + \vert y^{2,n}(t)-y^{1,n}(t) \vert \leq \inf_{\gamma_t} \Vert \gamma_t \Vert_{L^1(\R)},
\end{equation}
and 
\begin{equation} \label{test2}
\inf_{\gamma_0} \Vert \gamma_0 \Vert_{L^1(\R)}=\Vert \rho_0^{2,n}-\rho_0^{1,n} \Vert_{L^1(\R)} + \vert y_0^{2}-y_0^{1} \vert.
\end{equation}
For every $t\in [0,T]$, $\gamma_t$  is a piecewise elementary path, thus $\gamma_t$ admits wave shifts denoted by $\xi_i(t,\theta)$ and an SV shift denoted by $\xi_b(t,\theta)$. Therefore for a.e $\theta \in [0,1]$ and $t\in[0,T]$, 
\begin{equation} \label{test3}
\Vert  \gamma_t\Vert_{L^1(\R)}=\int_0^1 \sum_{k} \vert \Delta \rho^n_k(t,\theta)\xi_k^n(t,\theta)\vert + \vert \xi^n_b(t,\theta) \vert \, d\theta,  
\end{equation}
where $\Delta \rho^n_k(t,\theta)$ are the signed strengths of the $k^{\text{th}}$-waves. Thanks to \eqref{test1}, \eqref{test2} and  \eqref{test3}, to prove inequality \eqref{popo} it is enough to show, for every $\theta \in [0,1]$,
\begin{equation} \label{gg}
 \sum_{k} \vert \Delta \rho^n_k(T,\theta)\xi_k^n(T,\theta)\vert + \vert \xi^n_b(T,\theta) \vert  \leq C \left(\sum_{k} \vert \Delta \rho^n_k(0,\theta)\xi_k^n(0,\theta)\vert + \vert \xi^n_b(0,\theta) \vert \right),
\end{equation}
with $C>0$ independent of $n$.
 To simplify the notations, we drop the dependence on $\theta$ in \eqref{gg}. The following sections are devoted to the proof of \eqref{gg}.
 
\subsection{Introduction of $K(n,t_1,t_2,k)$} \label{110}
 In the sequel,  $\rho_L^k$ (resp.  $\rho_R^k$) denotes the car density at the left side (resp. at right side) of a $k^{\text{th}}$-discontinuity. Our goal in this section is to track the ancestors of a discontinuity along a wave-front tracking solution,
 without taking account interactions with the SV trajectory. 
 \begin{definition} \label{class}
We define the following waves and interactions:
\begin{itemize}
\item A classical shock $(\rho_l,\rho_r)$ is \textbf{either} a discontinuity such that $\rho_l<\rho_r$ (shock) \textbf{or} a discontinuity such that $\rho_r<\rho_l$ and $\rho_l-\rho_r\leq \frac{3}{2^{n+1}}$ (rarefaction).
\item A non classical shock $(\rho_l,\rho_r)$  is a discontinuity such that $\rho_r=\check \rho_{\alpha}$ and $\rho_l=\hat \rho_{\alpha}$. A non classical shock
can appear only along the SV trajectory.
\item A wave-wave interaction is an interaction between two waves away from the SV trajectory.
\item A wave-SV interaction is an interaction between a wave and the SV trajectory without creating or cancelling a non classical shock.
\end{itemize}
\end{definition}
\begin{definition} We now define the concept of ancestor.
\begin{itemize}
\item $K(n,t)$ denotes the set of classical shocks at time $t$. 
\item $i \in K(n,t_1)$ is an ancestor of  $j\in K(n,t_2)$ if $t_1 \leq t_2$ and $i$ can be connected by to $j$ via waves produced by wave-front tracking via interactions.
\item Let $0\leq t_1\leq t_2$. The set $K(n,t_1,t_2,k)$  denotes the set of classical shocks at time $t_1$ which are the ancestors of the $k^{\text{th}}$-wave with $k\in K(n,t_2)$ (see Example \ref{easy}). Moreover, we have $K(n,t_2,t_2,k)=\{k\}$.
\end{itemize}
\end{definition}

The following Lemma gives some basic properties of $K(n,t)$ and $K(n,t_1,t_2,k)$.
\begin{lemma} \label{entropy}
For every $0<t_1\leq t_2$, the following holds.
\begin{itemize}
\item Let $k\in K(n,t_2)$ and  $j\in K(n,t_2)\backslash \{k\}$. For every $(p,q)\in K(n,t_1,t_2,j)\times K(n,t_1,t_2,k)$, we have $p \neq q$.
\item $ \vert K(n,t_1,t_2,j)\vert  \leq  \vert K(n,t_0,t_2,j) \vert$  for every $0<t_0 \leq t_1$, where $\vert A \vert $ denotes the cardinality of the set $A$.
\item$K(n,t_1)=\sqcup_{k\in K(n,t_2)} K(n,t_1,t_2,k).$
\end{itemize}
\end{lemma}
\begin{proof}
We analyze the effect of wave interactions on $K$,
then all claims follow immediately.
\begin{itemize}
\item If no interaction occurs over $[t_1,t_2]$, then $K(n,t_2)=K(n,t)$ and $K(n,t,t_2,j)=\{j\}$ for every $(t,j)\in [t_1,t_2]\times K(n,t_2)$.

\item If a wave $1$ interacts with a wave $2$ creating a wave $3$  and no other interaction occurs at $t=\bar t>0$ \footnote{Since the lax entropy condition is verified when $\rho_L<\rho_R$ and no centered rarefaction wave can be created at $t>0$, the interaction presented in Figure \ref{2W} is the only one which can occur at $t>0$ far from the SV trajectory} (see Figure \ref{2W}) then  $K(n,\bar t^-)=\left(K(n,\bar t^+)\backslash \{3\}\right) \cup  \{1,2\}$ and $K(n,\bar t^-,\bar t^+,3)=\{1,2\}$.
\item The possible interactions of a wave the SV trajectory are presented in  Figure \ref{modbus}, Figure \ref{NSright}, Figure \ref{NSleft}  and Figure \ref{left1}. By definition, a non classical shock does not belong to $K(n,\cdot)$. Thus, we have $K(n,\bar t^+)=K(n,\bar t^-)$ and $K(n,\bar t^-,\bar t^+,j)=\{j\}$ with $j\in K(n,\bar t^+)$.
\end{itemize}
\end{proof}
\begin{definition}Let $\prec_n$ be the partial ordered over $\N$ defined as follows: $j \prec_n k$ if there exists $t_1\leq t_2$ such that $k\in K(t_2,n)$, $j\in K(n,t_1)$ and $j\in K(n,t_1,t_2,k)$. That is to say, $j$ is an ancestor of $k$ if and only if $j \prec_n k$.
\end{definition}
\begin{remark} If a centered rarefaction fan is created at $t=0$ (the sequence of discontinuities are denoted by $k_1,\cdots,k_m$), we have  $\xi^n_{k_i}=\xi^n_{k_{i+1}}$ for every $i=\{1,\cdots,m-1\}$ where $\xi^n_k$ denotes the shift of the $k^{\text{th}}$-wave. \\ 
 \end{remark}
\begin{example} \label{easy} In the particular case presented in Figure \ref{2W1}, we have
\begin{itemize}
\item  $K(n,T)=\{14,19\}$,
\item $K(n,t_3^+)=K(n,t_3^-)=\{2,3,4,5,6,9,19\}$,
\item $K(n,t_2^+)=\{2,3,4,5,6,9,19\}$ and $K(n,t_2^-)=\{2,3,4,5,6,9,17,18\}$,
\item $K(n,t_1^+)=\{2,3,4,5,6,7,8,17,18\}$ and $K(n,t_1-)=\{2,3,4,5,6,7,8,15,16,18\}$,
\item $K(n,0)=\{1,7,8,15,16,18\}$.
\end{itemize}
For instance, $K(n,t_2^+,T,14)=\{2,3,4,5,6,9\}$ or $K(n,t_1^+,T,19)=\{17,18\}.$ Moreover, $2 \prec_n 12$ and $7 \prec_n 12$ but $2$ and $7$ are not comparable.   \ \\
\begin{figure}
  \centering
      \includegraphics[width=1\linewidth]{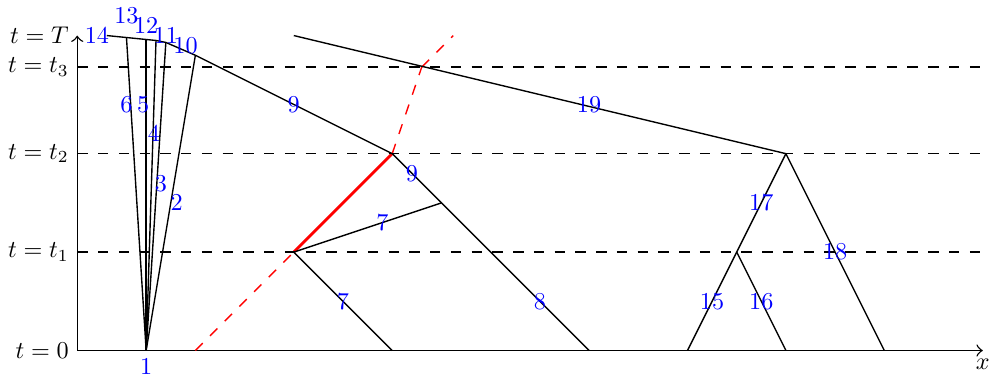}
   \caption{ \label{2W1} A particular configuration of discontinuities with the SV trajectory.}
\end{figure}
\end{example}
\subsection{Wave-wave  or wave-SV interactions } \label{111}
Let $n\in \N^*$ and $T>0$.
We describe all the possible interactions between two waves (see Figure \ref{2W}) and all the possible interactions between a wave and the SV trajectory (Figure \ref{modbus}, Figure \ref{NSright}, Figure \ref{NSleft}  and Figure \ref{left1}). There is no other possible interaction (for more details, we refer to \cite{DP14}). For each interaction, we determine the evolution of the shifts $\xi_k^n$ and the shift of the SV $\xi_b^n$ over time (see Lemma \ref{bb}, Lemma \ref{cc}, Lemma \ref{dd} and Lemma \ref{ee}). 
Since $\xi_k^n$ and $\xi_b^n$ remain constant if no interaction takes place (see \cite{BCP00,GP06}), we can only focus on wave-wave interactions and wave-SV  interactions. We introduce the function $\psi$ defined by
$$\psi(\rho_R,\rho_L)=\left\{ \begin{array}{ll}\frac{w(\rho_L)-w(\rho_R)}{w(\rho_L)-\lambda} & \text{if} \, \, (\rho_R,\rho_L) \in (\rho^*,1)\times ([0,\check{p}_\alpha]\cup[\hat{p}_\alpha,1]), \\ 
0 & \text{otherwise,}
\end{array}\right.
$$ with $\lambda:=1-\rho_L-\rho_R$ and $w$ defined in \eqref{velocity}.
We notice that $1-\psi(\rho_R,\rho_L)=\frac{w(\rho_R)-\lambda}{w(\rho_L)-\lambda}$ for $(\rho_R,\rho_L) \in (\rho^*,1)\times ([0,\check{p}_\alpha]\cup[\hat{p}_\alpha,1])$. By straightforward computations, we have 
\begin{equation} \label{ll}
\psi(\rho_L,\rho_R)=\left\{ \begin{array}{cl}
\frac{\rho_R-\rho^*}{\rho_L+\rho_R-\rho^*} &  \text{if} \quad (\rho_R>\rho^*  \, \, \& \, \,  \rho_l\in [0,\check{\rho}_\alpha]\cup[\hat{\rho}_\alpha,\rho^*]),\\
\frac{\rho_R-\rho_L}{\rho_R} & \text{if} \quad (\rho_R>\rho^* \, \, \& \, \, \rho_l\in [\rho^*,\rho_R]) \quad \text{or} \quad  (\rho^* \leq \rho_R <\rho_L), \\
0 & \text{otherwise.}
\end{array}\right.
\end{equation}
\begin{figure}
  \centering
      \includegraphics[width=0.25\linewidth]{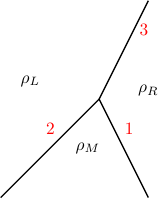}
   \caption{ \label{2W} Two waves interact together producing a third wave}
\end{figure}
\begin{figure}
\centering
\begin{minipage}[c]{0.4\textwidth}
  \centering
  \includegraphics[width=.7\linewidth]{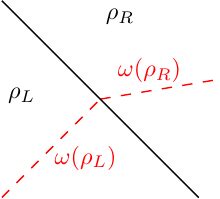}
  \caption*{Case a)  $\rho^*\leq \rho_R <\rho_L$ and $\rho_L-\rho_R \leq \frac{3}{2^{n+1}}$.} 
\end{minipage}\hfill
\begin{minipage}[c]{.47\textwidth}
  \centering
  \includegraphics[width=0.5\linewidth]{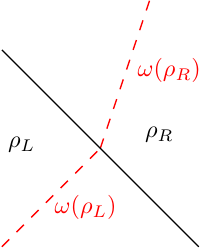}
  \caption*{Case b)  $\rho^*< \rho_R$ and $\rho_L \in [0,\check{\rho}_\alpha]  \cup [\hat{\rho}_\alpha,\rho_R]$.}
\end{minipage}
\caption{\label{modbus} Interaction coming from the right with the SV trajectory}
\end{figure}

\begin{figure}
   \begin{minipage}[c]{.5\linewidth}
   \centering
      \includegraphics[width=0.55\linewidth]{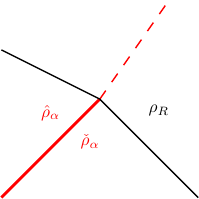}
      \caption*{Case a) $\rho_L=\check{\rho}_\alpha$ and $\rho_R \in [\hat{\rho}_\alpha,1]$}
   \end{minipage} \hfill
   \begin{minipage}[c]{.5\linewidth}
      \includegraphics[width=0.55\linewidth]{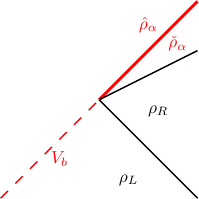}
      \centering
  \caption*{Case b) $\rho_L=\hat{\rho}_\alpha$ and $\rho_R \in [\check{\rho}_\alpha,\hat{\rho}_\alpha]$}
   \end{minipage}

   \caption{ \label{NSright} Interaction coming from the right with the SV trajectory  cancelling (Case a)) or creating (Case b)) a non classical shock.}
\end{figure}

\begin{figure}
   \begin{minipage}[c]{.5\linewidth}
   \centering
   \includegraphics[width=0.7\linewidth]{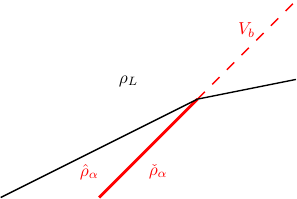}
      \caption*{Case a) $\rho_R=\hat{\rho}_\alpha$ and $\rho_L \in [0, \check{\rho}_\alpha]$}
   \end{minipage} \hfill
   \begin{minipage}[c]{.5\linewidth}
         \includegraphics[width=0.7\linewidth]{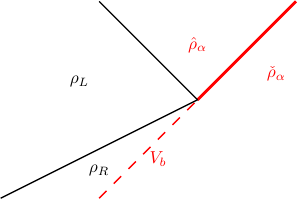}
      \centering
  \caption*{Case b)  $\rho_L \in [\check{\rho}_\alpha,\hat{\rho}_\alpha]$ and $\rho_R=\check{\rho}_\alpha$}
   \end{minipage}
  
   \caption{\label{NSleft} Interaction coming from the left with  the SV trajectory  cancelling (Case a)) or creating (Case b)) a non classical shock.}
\end{figure}

\begin{figure}
  \centering
      \includegraphics[width=0.35\linewidth]{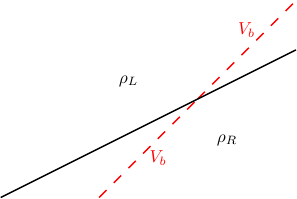}
   \caption{\label{left1}  $\rho_L\in[0,\check{\rho}_\alpha]$, $\rho_R \in [0,\check{\rho}_\alpha]\cup[\hat{\rho}_\alpha,p^*]$ and $\rho_L+\rho_R<\rho^*$. Interaction coming from the left with   the SV trajectory.}
\end{figure}
We introduce the function $\sigma$ defined by $\sigma(\rho_L,\rho_R):= \frac{f(\rho_L)-f(\rho_R)}{\rho_L-\rho_R}$ which represents the speed of the shock wave $(\rho_L,\rho_r)$. 

The following lemma, proved  in \cite[Lemma 2.7.2]{GP06}, deals with the interaction between two waves away from the SV trajectory (see Figure \ref{2W}). 
\begin{lemma} \label{123}
 The interaction between two waves produces a third wave.
\end{lemma}
\begin{proof}
We assume that a wave $(\rho_l,\rho_m)$ interacts with an other wave  $(\rho_m,\rho_r)$ without producing a third wave. Then, we have $\rho_l=\rho_r$. Since, in the wave-front tracking method, we have chosen that the speed of a rarefaction $(\rho_L,\rho_R)$ is $\sigma(\rho_L,\rho_R):= \frac{f(\rho_L)-f(\rho_R)}{\rho_L-\rho_R}$, the speed of the wave  $(\rho_l,\rho_m)$ is equal to the speed of the wave  $(\rho_m,\rho_r)$. We conclude that no interaction occurs, whence the contradiction.
\end{proof}
\begin{lemma}
\label{bb}
Consider two waves with speeds $\lambda_1$ and $\lambda_2$, respectively, that interact together at time $t=\bar t$ producing a wave with speed $\lambda_3$ (see Figure \ref{2W}). If the first wave is shifted by $\xi_1$ and the second wave by $\xi_2$, then the shift of the resulting wave is given by 
$$\xi_3=\frac{\lambda_3-\lambda_2}{\lambda_1-\lambda_2}\xi_1+\frac{\lambda_1-\lambda_3}{\lambda_1-\lambda_2}\xi_2.$$
Besides,
$$\Delta \rho_3 \xi_3=\Delta \rho_2 \xi_2+\Delta \rho_1 \xi_1=\sum_{k\in K(n,\bar t^-,\bar t^+,3)} \Delta \rho_k \xi_k ,$$
with $\Delta \rho_3=\rho_R-\rho_L$,  $\Delta \rho_1=\rho_R-\rho_M$ and  $\Delta \rho_2=\rho_M-\rho_L$.
\end{lemma}
 Let $\bar t \in \R_+^*$.  The following Lemmas deal with the interaction between a wave $k$ and the SV trajectory with $k\in K(n,\bar t^+)=K(n,\bar t^-)$. Lemma \ref{cc} is proved in \cite[Section 4.2]{BP08}. The proof of Lemma \ref{dd} and Lemma \ref{ee} are standard and they are obtained  by mimicking the proof of  Lemma \ref{cc}. 
\begin{lemma}  We assume that the wave $k$ interacts at time $t=\bar t$ with the SV trajectory without creating or cancelling a non-classical shock (see Figure \ref{modbus}), then 
\label{cc}
$$\left\{ \begin{array}{l}\xi_b(\bar t^+)=(1-\psi(\rho_L^{k}(\bar t^-),\rho_R^{k}(\bar t^-)))\xi_b(\bar t^-)+\psi(\rho_L^{k}(\bar t^-),\rho_R^{k}(\bar t^-))\xi_{k}(\bar t^-),\\
\xi_k(\bar t^+)=\xi_k(\bar t^-), \\ \end{array} \right.$$
with $\xi_b$ the SV shift and $\xi_k$ the shift of the wave $k$. 
\end{lemma}
\begin{lemma} \label{dd} We assume that the  wave $k$ interacts with the SV trajectory and a non classical shock is created (see Figure \ref{NSright} (right) and Figure \ref{NSleft} (right)). Then, 
$$ \left\{ \begin{array}{l}
\xi_b(\bar t^+)=\xi_b(\bar t^-),\\
\Delta \rho_k(\bar t^+) \xi_{k}(\bar t^+)+(\check{\rho}_\alpha-\hat{\rho}_\alpha)\xi_{b}(\bar t^+)=\Delta \rho_k(\bar t^-) \xi_{k}(\bar t^-),\\
\end{array} \right.  $$
   with $\Delta \rho_k=\rho_R^k-\rho_L^k$, $\xi_b$ the SV shift and $\xi_k$ the  shift of the wave $k$.  
\end{lemma}
\begin{lemma} \label{ee} We assume that the wave $k$ interacts with the SV trajectory and a non classical shock is cancelled (see Figure \ref{NSright} \textcolor{red}{a)} and Figure \ref{NSleft} \textcolor{red}{a)})
$$ \left\{ \begin{array}{l}
\xi_b(\bar t^+)=(1-\psi(\rho_L^{k}(\bar t^-),\rho_R^{k}(\bar t^-))) \xi_b(\bar t^-)+\psi(\rho_L^{k}(\bar t^-),\rho_R^{k}(\bar t^-))\xi_{k}(\bar t^-),\\
\Delta \rho_k(\bar t^+) \xi_{k}(\bar t^+)=\Delta \rho_k(\bar t^-) \xi_{k}(\bar t^-)+(\check{\rho}_\alpha-\hat{\rho}_\alpha)\xi_{b}(\bar t^-),\\
\end{array} \right.  $$
with $\Delta \rho_k=\rho_R^k-\rho_L^k$, $\xi_b$ the SV shift and $\xi_k$ the shift of the wave $k$. 
\end{lemma}
The proof of Theorem \ref{main1} is based on the following estimates whose the proof is postponed in Appendix \ref{app::bababa}.
\begin{lemma} \label{bababa}
\begin{itemize}
\item[\textbf{A.}] $\vert(1-\psi(\rho_L,\rho_R)\vert\leq 1+\frac{3}{2^{n+1}\rho^*}< 1+\frac{2}{\rho^*}$, for every $(\rho_R,\rho_L) \in (\rho^*,1)\times ([0,\check{p}_\alpha]\cup[\hat{p}_\alpha,1])$,
\item[\textbf{B.}] $\vert \frac{\psi(\rho_L,\rho_R)}{\rho_R-\rho_L} \vert \leq \frac{2}{\rho^*}$, for every $(\rho_R,\rho_L) \in (\rho^*,1)\times ([0,\check{p}_\alpha]\cup[\hat{p}_\alpha,1])$,
\item[\textbf{C.}] $\vert \frac{\psi(\rho_L,\rho_R)}{\rho_R-\rho_L} \vert +\vert1-\psi(\rho_L,\rho_R)\vert\frac{2}{\rho^*} \leq \frac{2}{\rho^*}$, for every $(\rho_R,\rho_L) \in (\rho^*,1)\times  [0,\check{p}_\alpha]$.
\end{itemize}
\end{lemma}
We give a further property of non classical shocks.
 \begin{lemma}\label{test}
A wave interacting with a non classical shock cancels it and produces an outgoing classical shock wave. 
\end{lemma}
\begin{proof} 
We assume that a wave $1$ $(\rho_L^1,\rho_R^1)$ interacts at time $\bar t$ with a non classical shock $(\hat \rho_{\alpha},\check \rho_{\alpha})$ coming from the left. 
\begin{itemize}
\item If the wave 1 is reflected in the non classical shock; in this case, a wave $2$ $(\rho_L^2, \rho_R^2)$ is produced at $t=\bar t$ with $ \rho_L^1=\rho_L^2$. Moreover, by construction, we have $ \rho_R^1=\hat \rho_{\alpha}$ and $ \rho_R^2=\hat \rho_{\alpha}$. We conclude that $\sigma(\rho_L^1,\rho_R^1)=\sigma(\rho_L^2,\rho_R^2)$ which is not possible for a reflection.
\item If the wave goes through to the non classical shock; in this case $\rho_L^1=\hat \rho_{\alpha}=\rho_R^1$, which is obviously not possible.
\end{itemize}
We conclude that a wave, coming from the left of the SV trajectory, cancels a non classical shock and, from Lemma \ref{123}, the interaction  produces an outgoing classical shock wave. A similar proof can be done for a wave coming from the right. 
\end{proof}
\subsection{Proof of Theorem \ref{main1}} \label{jhjh}

\subsubsection{Ideas of the proof (a backwards in time method)} \label{explication}

For every $k\in K(n,T)$, we want to track the exact values of $\xi_b$ and $\xi_j$ with $j\in K(n,t,T,k)$ from $t=T$ until $t=0$.  We assume that no interactions occurs over $(t_1,T]$ and at $t=t_1$ we have either a wave-wave interaction (see Figure \ref{2W}) or a wave-SV interaction (see Figure \ref{modbus}, \ref{NSright}, \ref{NSleft} and Figure  \ref{left1}). For every $t\in (t_1 , T]$ and for every $k\in K(n,T)=K(n,t)$, we get $\Delta \rho^n_k(T)\xi_k^n(T) =  \Delta \rho^n_k(t)\xi_k^n(t)$ and $ \xi^n_b(T) = \xi^n_b(t)$.


\begin{itemize}
\item If the wave $1$ interacts with the wave $2$ producing the wave $3$  (see Figure \ref{2W}) then, from Lemma \ref {bb}, 
\begin{equation} \label{etape1}\left\{ \begin{array}{l}
\xi_b (T)=\xi_b(t_1^-),\\
\Delta \rho_k(T) \xi_k(T)=\Delta \rho_k(t_1^-) \xi_k(t_1^-), \quad \text{for every} \quad k\in K(n, T) \backslash \{3\}=K(n,t_1^-)\backslash \{1,2\},\\ 
  \Delta \rho_{3}(T)\xi_{3}(T)=\Delta \rho_{1}(t_1^-) \xi_{1}(t_1^-)+\Delta \rho_{2}(t_1^-) \xi_{2}(t_1^-).
  \end{array} \right. \end{equation}

\item If  the wave $k$ with $k\in K(n,T)$ interacts with the SV trajectory at $t=t_1$ (see Figures \ref{modbus}, \ref{NSright},  \ref{NSleft} and \ref{left1}) then, from Lemma \ref{cc}, \ref{dd} and Lemma \ref{ee}, there exists $(W_{b,k}^1(t_1), W_{b,k}^2(t_1), W_k^1(t_1), W_k^2(t_1)) \in \R_+^4$ such that $K(n,T)=K(n,t_1^-)$ and
\begin{equation} \label{etape2} \left\{ \begin{array}{l}
\xi_b(T)=W^1_{b,k}(t_1) \xi_b(t_1^-)+W^1_k(t_1)\Delta \rho_k(t_1^-) \xi_{k}(t_1^-),\\
\Delta \rho_{ k}(T) \xi_{ k}(T)=W^2_{b,k}(t_1)\xi_{b}(t_1^-)+W^2_{k}(t_1) \Delta \rho_{ k}(t_1^-) \xi_{ k}(t_1^-),\\
\Delta \rho_j(T) \xi_j(T)=\Delta \rho_j(t_1^-) \xi_j(t_1^-), \quad \text{for every} \, \, j\in K(n,T) \backslash \{ k\}=K(n,t_1^-) \backslash \{ k\},\\
\end{array} \right.  \end{equation}
\end{itemize}
At time $t=t_1>0$, we repeat the previous strategy considering $K(n,t_1^-)$ instead of $K(n,T)$ until a second interaction time $t=t_2>0$ and so on. Combining \eqref{etape1} with \eqref{etape2}, for every $k \in K(n,T)$, there exist  $W^1_b(0),W^2_{b,k}(0), (W^1_{j,k}(0))_{j\in K(n,0)}, (W^2_{j,k}(0))_{j\in K(n,0)} \in \R_+^2 \times \R^{\vert K(n,0)\vert} \times \R^{\vert K(n,0)\vert}$  such that 
\begin{equation} \label{but}
 \left\{ \begin{array}{l}
\xi_b(T)=W_b^1(0) \xi_b(0)+\sum_{j\in K(n,0)}W^2_{j,k}(0)\Delta \rho_j(0) \xi_{j}(0),\\
\Delta \rho_{k}(T) \xi_{k}(T)=W_{b,k}^{2}(0)\xi_{b}(0)+\sum_{j\in K(n,0)}W_{j,k}(0) \Delta \rho_{j}(0) \xi_{j}(0).\\
\end{array} \right.  \end{equation}
From \eqref{but}, we construct explicitly weight functions $(W^n_k(0))_{k\in K(n,0)}$ and $W^n_b(0)$  such that 
\begin{equation} \label{gg23}
 \sum_{k\in K(n,T)} \vert \Delta \rho^n_k(T)\xi_k^n(T)\vert + \vert \xi^n_b(T) \vert  \leq  \sum_{k\in K(n,0)} \vert W^n_k(0)\Delta \rho^n_k(0)\xi_k^n(0)\vert + \vert W^n_b(0) \xi^n_b(0) \vert.
\end{equation}
The desired inequality \eqref{gg} is obtained using the following Lemma
\begin{lemma} \label{ffff}
Let $T>0$ and $n\in \N^*$.  There exists $((W^n_k(0))_{k \in K(n,0)}, W^n_b(0))\in \R^{\vert K(n,0) \vert} \times (\R^*_+)$ such that  \begin{equation} \label{gg2}
 \sum_{k\in K(n,T)} \vert \Delta \rho^n_k(T)\xi_k^n(T)\vert + \vert \xi^n_b(T) \vert  \leq  \sum_{k\in K(n,0)} \vert W^n_k(0)\Delta \rho^n_k(0)\xi_k^n(0)\vert + \vert W^n_b(0) \xi^n_b(0) \vert ,
\end{equation}
with $$\max(\vert W^n_k(0),  W^n_b(0)) \leq C, \quad \text{for every} \, \, k \in K(n,0), $$
with $C>0$ a constant independent of $n$.
\end{lemma}
Lemma \ref{ffff} is proved by considering only the interactions  which may occur an infinite number of times with the SV trajectory. In theses cases, the wave and SV shifts may blow up.
In Section \ref{fefefe}, we study the evolution of the SV shift and the evolution of the wave shifts when a non classical shock is created and then cancelled as well as the number of times these particular types of interaction can occur successively. In Section \ref{sectionsection}, we determine the expressions of $W_k^n(0)$ and $W_k^n(0)$, defined in \eqref{gg2}, in the case where  the wave and SV shifts may blow up. Lemma \ref{bababa}  proves that $W_k^n(0)$ and $W_b^n(0)$ are bounded  independent of $n$.  \\

\textbf{Example:} We consider the particular case presented in Figure \ref{2W1}. From Lemma \ref{bb} and Lemma \ref{cc} and $\xi_b(T)=\xi_b(t_3^+)$, we have 
$$\left\{ \begin{array}{l}\xi_b( T)=(1-\psi(\rho_L^{19},\rho_R^{19})\xi_b(t_3^-)+\frac{\psi(\rho_L^{19},\rho_R^{19})}{\Delta \rho_{19}(t_3^-)}\Delta \rho_{19}(t_3^-)\xi_{19}(t_3^-),\\
\Delta \rho_{14}(T)\xi_{14}(T)=\sum_{k \in \{2,3,4,5,6,9\}}\Delta \rho_{k}(t_3^-)\xi_k(t_3^-),  \\
\Delta \rho_{19}(T)\xi_{19}(T)=\Delta \rho_{19}(t_3^-)\xi_{19}(t_3^-),\\
  \end{array} \right.$$
  with $\psi(\rho_L^{19},\rho_R^{19})=\psi(\rho_L^{19}(t_3^+),\rho_R^{19}(t_3^+))=\psi(\rho_L^{19}(t_3^-),\rho_R^{19}(t_3^-))$ and $\Delta \rho_{19}(t_3^-)=\rho_R^{19}-\rho_L^{19}$.
From  Lemma \ref{bb} and Lemma \ref{ee}, we get
$$ \left\{ \begin{array}{l}
\xi_b(t_2^+)=(1-\psi(\rho_L^{9}(t_2^-),\rho_R^{9}(t_2^-))) \xi_b(t_2^-)+\frac{\psi(\rho_L^{9}(t_2^-),\rho_R^{9}(t_2^-))}{\Delta \rho_{9}(t_2^-)}\Delta \rho_{9}(t_2^-)\xi_{9}(t_2^-),\\
\Delta \rho_9(t_2^+) \xi_{9}(t_2^+)=\Delta \rho_9(t_2^-) \xi_{9}(t_2^-)+(\check{\rho}_\alpha-\hat{\rho}_\alpha)\xi_{b}(t_2^-),\\
\Delta \rho_{19}(t_2^+) \xi_{19}(t_2^+)=\Delta \rho_{17}(t_2^-) \xi_{17}(t_2^-)+\Delta \rho_{18}(t_2^-) \xi_{18}(t_2^-).\\
\end{array} \right.  $$
We notice that $\rho_L^{9}(t_2^-)=\check \rho_{\alpha}$ and  $\rho_R^{9}(t_2^+)=\rho_R^{9}(t_2^-)$. Since $\xi_b(t_3^-)=\xi_b(t_2^+)$,
\begin{displaymath}  
\left\{ \begin{array}{l}\xi_b( T)=(1-\psi(\rho_L^{19},\rho_R^{19}))(1-\psi(\rho_L^{9}(t_2^-),\rho_R^{9}(t_2^-)))\xi_b(t_2^-), \\
\, \quad + \,(1-\psi(\rho_L^{19},\rho_R^{19}))\frac{\psi(\rho_L^{9}(t_2^-),\rho_R^{9}(t_2^-))}{\Delta \rho_{9}(t_2^-)}\Delta \rho_{9}(t_2^-)\xi_{9}(t_2^-) +\frac{\psi(\rho_L^{19},\rho_R^{19})}{\Delta \rho_{19}(t_3^-)}\sum_{k\in \{17,18\}}\Delta \rho_{k}(t_2^-)\xi_{k}(t_2^-),\\
\Delta \rho_{14}(T)\xi_{14}(T)=\sum_{k \in \{2,3,4,5,6,9\}}\Delta \rho_{k}(t_2^-)\xi_k(t_2^-)+(\check{\rho}_\alpha-\hat{\rho}_\alpha)\xi_{b}(t_2^-),  \\
\Delta \rho_{19}(T)\xi_{19}(T)=\sum_{k \in \{17,18\}}\Delta \rho_{k}(t_2^-)\xi_k(t_2^-).\\
  \end{array} \right. \end{displaymath}
 From Lemma \ref{bb} and Lemma \ref{dd},
$$ \left\{ \begin{array}{l}
\xi_b(t_1^+)=\xi_b(t_1^-),\\
\Delta \rho_7(t_1^+) \xi_{7}(t_1^+)+(\check{\rho}_\alpha-\hat{\rho}_\alpha)\xi_{b}(t_1^+)=\Delta \rho_7(t_1^-) \xi_{7}( t_1^-),\\
\Delta \rho_{17}(t_1^+) \xi_{17}(t_1^+)=\Delta \rho_{15}(t_1^-) \xi_{15}(t_1^-)+\Delta \rho_{16}(t_1^-) \xi_{16}(t_1^-).
\end{array} \right.  $$
Since $\xi_b(t_2-)=\xi_b(t_1^+)$ and $\Delta \rho_{9}(t_2^-)\xi_{9}(t_2^-)=\Delta \rho_{7}(t_1^+)\xi_{7}(t_1^+)+\Delta \rho_{8}(t_1^+)\xi_{8}(t_1^+)$, 
$$\left\{ \begin{array}{l}\xi_b( T)=(1-\psi(\rho_L^{19},\rho_R^{19}))(1-\psi(\rho_L^{9}(t_2^-),\rho_R^{9}(t_2^-)))\xi_b(t_1^-), \\
+ (1-\psi(\rho_L^{19},\rho_R^{19}))\frac{\psi(\rho_L^{9}(t_2^-),\rho_R^{9}(t_2^-))}{\Delta \rho_{9}(t_2^-)}\sum_{k=7}^8\Delta \rho_{k}(t_1^-)\xi_{k}(t_1^-)+\frac{\psi(\rho_L^{19},\rho_R^{19})}{\Delta \rho_{19}(t_3^-)}\sum_{k\in \{15,16,18\}}\Delta \rho_{k}(t_1^-)\xi_{k}(t_1^-),\\
\Delta \rho_{14}(T)\xi_{14}(T)=\sum_{k \in \{2,3,4,5,6,7,8\}}\Delta \rho_{k}(t_1^-)\xi_k(t_1^-),  \\
\Delta \rho_{19}(T)\xi_{19}(T)=\sum_{k \in \{15,16,18\}}\Delta \rho_{k}(t_1^-)\xi_k(t_1^-).\\
  \end{array} \right.$$
By convention, we have $\xi_1(0)=\xi_k(0^+)$ for every $k\in \{2,3,4,5,6\}$. Using $\xi_b(t_1+)=\xi_b(0)$, we conclude that 
$$\left\{ \begin{array}{l}\xi_b( T)=(1-\psi(\rho_L^{19},\rho_R^{19}))(1-\psi(\rho_L^{9}(t_2^-),\rho_R^{9}(t_2^-)))\xi_b(0), \\
+ \,(1-\psi(\rho_L^{19},\rho_R^{19}))\frac{\psi(\rho_L^{9}(t_2^-),\rho_R^{9}(t_2^-))}{\Delta \rho_{9}(t_2^-)}\sum_{k\in \{7,8\}}\Delta \rho_{k}(0)\xi_{k}(0)+\frac{\psi(\rho_L^{19},\rho_R^{19})}{\Delta \rho_{19}(t_3^-)}\sum_{k\in \{15,16,18\}}\Delta \rho_{k}(0)\xi_{k}(0),\\
\Delta \rho_{14}(T)\xi_{14}(T)=\sum_{k \in \{1,7,8\}}\Delta \rho_{k}(0)\xi_k(0),  \\
\Delta \rho_{19}(T)\xi_{19}(T)=\sum_{k \in \{15,16,18\}}\Delta \rho_{k}(0)\xi_k(0).\\
  \end{array} \right.$$
In the particular case presented in Figure \ref{2W1}, the inequality \eqref{gg} becomes 
$$
 \sum_{k\in \{14,19\} } \vert \Delta \rho^n_k(T)\xi_k^n(T)\vert + \vert \xi^n_b(T) \vert  \leq  \sum_{k\in \{1,7,8,15,16,18\}}  W^n_k(0)\vert\Delta \rho^n_k(0)\xi_k^n(0)\vert +  W^n_b(0)\vert \xi^n_b(0) \vert,
$$
with 
 $$ \left\{ \begin{array}{l}
 W_b(0)= \vert(1-\psi(\rho_L^{19},\rho_R^{19}))(1-\psi(\rho_L^{9}(t_2^-),\rho_R^{9}(t_2^-))) \vert, \\
 W_1(0)= 1,  \\
 W_7(0)=W_8(0)= 1+\vert (1-\psi(\rho_L^{19},\rho_R^{19}))\frac{\psi(\rho_L^{9}(t_2^-),\rho_R^{9}(t_2^-))}{\Delta \rho_{9}(t_2^-)} \vert, \\
  W_{15}(0)=W_{16}(0)=W_{18}(0)=  1+\vert \frac{\psi(\rho_L^{19},\rho_R^{19})}{\Delta \rho_{19}(t_3^-)} \vert.   \\
   \end{array}
 \right.
 $$
Applying Lemma \ref{bababa}, we obtain 
$$\max(\vert W^n_k(0),  W^n_b(0)) \leq 1+\left(1+\frac{2}{\rho^*}\right)^2, \quad \text{for every} \, \, k\in \{1,7,8,15,16,18\} , $$


\begin{remark}
 We assume that the wave $k_1$ interacts with the SV trajectory at $t=t_1$ creating a non classical shock (Lemma \ref{dd}) and  the wave $k_2$ interacts with the SV trajectory at $t=t_2$ cancelling the previous non classical shock (Lemma \ref{ee})). Using $\xi_b(t_1^+)=\xi_b(t_2^-)$, by straighforward computations we have
$$\Delta \rho_{k_1}(t_1^+) \xi_{k_1}(t_1^+)+\Delta \rho_{k_2}(t_2^+) \xi_{k_2}(t_2^+)=\Delta \rho_{k_1}(t_1^-) \xi_{k_1}(t_1^-)+\Delta \rho_{k_2}(t_2^-) \xi_{k_2}(t_2^-)+ \textbf{0}\, \xi_b(t_1^+),$$  and
$$\vert \Delta \rho_{k_1}(t_1^+) \xi_{k_1}(t_1^+)\vert +\vert \Delta \rho_{k_2}(t_2^+) \xi_{k_2}(t_2^+)\vert \leq \vert \Delta \rho_{k_1}(t_1^-) \xi_{k_1}(t_1^-)\vert +\vert \Delta \rho_{k_2}(t_2^-) \xi_{k_2}(t_2^-)\vert +\vert  2(\check{\rho}_\alpha-\hat{\rho}_\alpha)\vert \vert \xi_b(t_1^+)\vert .$$ Since this type of interactions can occur an infinite number of times, the usual locally method, which consists in constructing a weight function $W_b$ for the SV shift and weight functions $W_k$ for waves such that $t \mapsto W_b(t) \vert \xi_b(t) \vert +W_k(t) \vert \Delta \rho_k(t) \xi_k(t) \vert $ for every $k\in K(n,t)$ are not increasing in time (see \cite{BS99,CRM03,M14}), is more difficult to apply. That is why we introduce a backward in time method described above which captures all information over $ [0,T]$. 
\end{remark}
\begin{remark}
Since $k^{\text{th}}$-wave may interact an infinite number of times with the SV trajectory, to find an upper bound of the weight $W_k^n(0)$, we have to prove that an infinite serie is bounded, which is the difficult point of this proof (see Proof of Lemma \ref{coco}). 
\end{remark}
To obtain \eqref{but}, we need to have a better understandable of the creation and cancellation of a non classical shock (see section \ref{fefefe}).
    
\subsubsection{Creation and cancellation of a non classical shock} \label{fefefe}
We assume that a non classical shock is created at $t=t_1>0$. Let $t_2$ the first time after $t_1$ where a wave interacts with the SV trajectory. From Lemma \ref{test}, the non classical shock is cancelled at time $t_2$. We have two possibilities to create a non classical shock (see Figure \ref{NSright} \textcolor{red}{b)} and Figure \ref{NSleft} \textcolor{red}{b)}) and we have two possibilities to cancel a non classical shock (see Figure \ref{NSright} \textcolor{red}{a)} and Figure \ref{NSleft} \textcolor{red}{a)}). Thus, we have four types of interaction denoted by 
\textbf{\textcolor{blue}{NC1}}, \textbf{\textcolor{blue}{NC2}}, \textbf{\textcolor{blue}{NC3}} and \textbf{\textcolor{blue}{NC4}} which will be described below.  For each possible case, we determine the evolution of the tangent vectors backwards in time. 


\begin{itemize} 
  \item[\textbf{\textcolor{blue}{NC1:}}]  A \textbf{\textcolor{blue}{NC1}} interaction is obtained combining Figure \ref{NSright} \textcolor{red}{b)}  with Figure \ref{NSleft} \textcolor{red}{a)}; the wave $1$ \\
   $(\rho_L^{1}(t_1^-),\rho_R^{1}(t_1^-))$ interacts with the SV trajectory at time $t_1$ creating a non classical shock and a wave $(\rho_L^{1}(t_1^+),\rho_R^{1}(t_1^+))$ which will be  denoted by $1$ as well. The wave $2$ $(\rho_L^{2}(t_2^-),\rho_R^{2}(t_2^-))$ interacts with the SV trajectory at time $t_2$ cancelling a non classical shock and producing a wave $(\rho_L^{2}(t_2^+),\rho_R^{2}(t_2^+))$ which will be called $2$ as well (see Figure \ref{CN3} \textcolor{red}{a)}). In this case, 
$  \rho_L^{1}(t_1^-)=\hat \rho_{\alpha}$, $\rho_R^{1}(t_1^-)=\rho_R^{1}(t_1^+) \in [\check \rho_{\alpha},\hat \rho_{\alpha}]$, $\rho_L^1(t_1^+)= \check \rho_{\alpha}$ and $\rho_L^{2}(t_2^-)= \rho_L^{2}(t_2^+)\in [0,\check \rho_{\alpha}]$, $\rho_R^{2}(t_2^-)=\hat \rho_{\alpha}$, $\rho_R^{2}(t_2^+)=\check \rho_{\alpha}$.
     \begin{figure}
 \begin{minipage}[c]{.5\linewidth}
   \centering
      \includegraphics[width=0.6\linewidth]{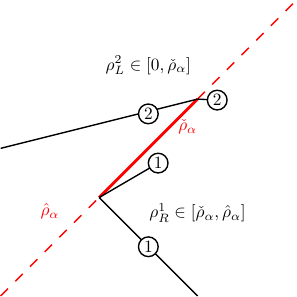}
      \caption*{Case $a)$}
   \end{minipage} \hfill
   \begin{minipage}[c]{.5\linewidth}
   \centering
      \includegraphics[width=0.6\linewidth]{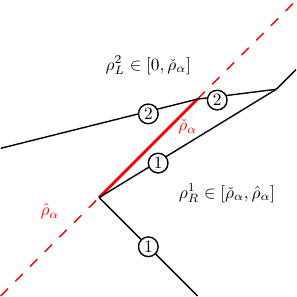}
      \caption*{Case $b)$}
   \end{minipage} \hfill
   \begin{minipage}[c]{.5\linewidth}
      \includegraphics[width=0.6\linewidth]{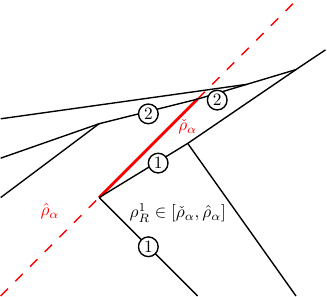}
      \centering
  \caption*{Case $c)$}
   \end{minipage}
  \caption{ \label{CN3} Different examples of   \textcolor{blue}{NC1} interactions}
\end{figure} From  Lemma \ref{dd} and Lemma \ref{ee}, we have 
 \begin{equation} \label{paaa}
 \left\{\begin{array}{l}
 \xi_b(t_2^+)=(1-\psi(\rho_L^{2}(t_2^-),\rho_R^{2}(t_2^-))) \xi_b(t_2^-)+\frac{\psi(\rho_L^{2}(t_2^-),\rho_R^{2}(t_2^-))}{\Delta \rho_2(t_2^-)}\Delta \rho_2(t_2^-)\xi_2(t_2^-),\\
\Delta \rho_2(t_2^+) \xi_{2}(t_2^+)=\Delta \rho_2(t_2^-) \xi_{2}(t_2^-)+(\check{\rho}_\alpha-\hat{\rho}_\alpha)\xi_{b}(t_2^-),\\
 \xi_b(t_1^+)=\xi_b(t_1^-),\\
\Delta \rho_1(t_1^+) \xi_{1}(t_1^+)+(\check{\rho}_\alpha-\hat{\rho}_\alpha)\xi_{b}(t_1^+)=\Delta \rho_1(t_1^-) \xi_{1}(t_1^-),\\
 \end{array}\right.
 \end{equation}
 with $\Delta \rho_{2}(t_2^-):= \rho_R^{2}(t_2^-)-\rho_L^{2}(t_2^-)$, $\Delta \rho_{1}(t_1^-):= \rho_R^{1}(t_1^-)-\rho_L^{2}(t_1^-)$, $\Delta \rho_{1}(t_1^+):= \rho_R^{1}(t_1^+)-\rho_L^{1}(t_1^+)$ and $\psi$ is defined in \eqref{ll}.

From Lemma \ref{test} and by definition of $t_2$, no other wave can interact with the non classical shock over $[t_1,t_2]$. Thus,  we have  $\xi_b(t_2^-)=\xi_b(t_1^+)$ (see Figure \ref{CN3} \textcolor{red}{a)}) and $\Delta \rho_2(t_2^-) \xi_{2}(t_2^-)=\sum_{k\in K(n,t_1^-,t_2^+,2)}\Delta \rho_k \xi_k$. Moreover, since $\rho_L^{2}(t_2^-)=\hat \rho_{\alpha}<\rho^*$, we have $\psi(\rho_L^{2}(t_2^-),\rho_R^{2}(t_2^-))=0$. Using \eqref{paaa}, we conclude that 
 \begin{equation} \label{panda22}
 \left\{\begin{array}{l}
 \xi_b(t_2^+)= \xi_b(t_1^-),\\
\Delta \rho_2(t_2^+) \xi_{2}(t_2^+)=\sum_{k\in K(n,t_1^-,t_2^+,2)}\Delta \rho_k \xi_k+(\check{\rho}_\alpha-\hat{\rho}_\alpha)\xi_{b}(t_1^-),\\
\Delta \rho_1(t_1^+) \xi_{1}(t_1^+)+(\check{\rho}_\alpha-\hat{\rho}_\alpha)\xi_{b}(t_1^-)=\Delta \rho_1(t_1^-) \xi_{1}(t_1^-).\\
 \end{array}\right.
 \end{equation}
We distinguish two different cases. By construction, there exists a couple $(k_1,k_2) \in K(n,T)^2$ such that $1\in K(n,t_1^+,T,k_1)=K(n,t_1^-,T,k_1)$ and $2 \in K(n,t_2^+,T,k_2)=K(n,t_2^-,T,k_2)$.
 \begin{itemize}
    \item[\textbf{\textcolor{blue}{NC1-a)}}] $k_1 \neq k_2$; roughly speaking the wave $1$  never interacts with the wave $2$. Let $k_0 \in K(n,t_2^+,T,k_1)$ such that $1\in K(n,t_1^+,t_2^+,k_0)$. From  \eqref{panda22}, we have
\begin{equation} \label{casrare} 
 \left\{\begin{array}{l}
 \xi_b(t_2^+)= \xi_b(t_1^-),\\
\Delta \rho_2(t_2^+) \xi_{2}(t_2^+)=\sum_{k\in K(n,t_1^-,t_2^+,2)}\Delta \rho_k \xi_k+(\check{\rho}_\alpha-\hat{\rho}_\alpha)\xi_{b}(t_1^-),\\
\Delta \rho_{k_{0}}(t_2^+) \xi_{k_0}(t_2^+)=\sum_{k\in K(n,t_1^-,t_2^+,k_0)}\Delta \rho_k \xi_k-(\check{\rho}_\alpha-\hat{\rho}_\alpha)\xi_{b}(t_1^-).\\
 \end{array}\right.
 \end{equation} 

  \begin{lemma} \label{rare}
 A wave coming from the right  cannot interact with the SV trajectory  at time $t>t_2$. In particular, a \textbf{\textcolor{blue}{NC1-a)}} interaction occurs at most one time.
  \end{lemma}
  \begin{proof}
   Assuming a wave $k$, coming from the right,  interacts with the SV trajectory at $t>t_2$. Using Lemma \ref{123}, we have $1 \prec_n k$ and $2 \prec_n k$. Thus, $k_1=k_2$, whence the contradiction. Since $\textbf{\textcolor{blue}{NC1-a)}}$ starts with an interaction coming from the right of the SV trajectory, the case $\textbf{\textcolor{blue}{NC1-a)}}$ can happen only once. 
  
  \end{proof}

 \item[\textbf{\textcolor{blue}{NC1-b)}}] $k_1=k_2$; two different types of  \textbf{\textcolor{blue}{NC1-b)}} interaction are illustrated   in Figure \ref{CN3} \textcolor{red}{b)} and  in Figure \ref{CN3} \textcolor{red}{c)}.  Roughly speaking the wave $1$  interacts with the wave $2$. Since $t_2$ is the first time after $t_1$ where a wave interacts with the SV trajectory and $k_1=k_2$, there exist $t_3 \in(t_2,T]$ and a wave $3 \in K(n,t_3^+,T,k_1)$ such that every wave $k \in K(n,t_1^+,t_3^+,3)\backslash \{2\}$ does not interact with the SV trajectory. Thus, we have 
 \begin{equation} \label{aaa}\Delta \rho_3(t_3^+) \xi_3(t_3^+)=\Delta \rho_1(t_1^+) \xi_1(t_1^+)+\Delta \rho_2(t_2^+) \xi_2(t_2^+) + \sum_{k\in K(n,t_1^+,t_3^+,3)\backslash \left(K(n,t_1^+,t_2^+,2)\cup\{1\}\right)} \Delta \rho_k \xi_k. \end{equation}
 From \eqref{panda22}, \eqref{aaa} and the equality $\Delta \rho_2(t_2^-) \xi_2(t_2^-)=\sum_{k\in K(n,t_1^-,t_2^-,2)} \Delta \rho_k \xi_k$, we conclude that
 \begin{equation} \label{panda2229}
 \left\{\begin{array}{l}
 \xi_b(t_2^+)= \xi_b(t_1^-),\\
 \Delta \rho_3(t_3^+) \xi_{3}(t_3^+)=\sum_{k\in K(n,t_1^-,t_3,3)} \Delta \rho_k \xi_k.
 \end{array}\right.
 \end{equation}
 From \eqref{panda2229}, a \textbf{\textcolor{blue}{NC1-b)}} interaction has the same effect as wave-wave interations.
 \end{itemize}
  \begin{figure}
   \begin{minipage}[c]{.5\linewidth}
   \centering
      \includegraphics[width=0.6\linewidth]{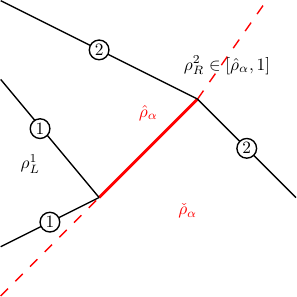}
      \caption*{Case $a)$}  
   \end{minipage} \hfill
   \begin{minipage}[c]{.5\linewidth}
      \includegraphics[width=0.6\linewidth]{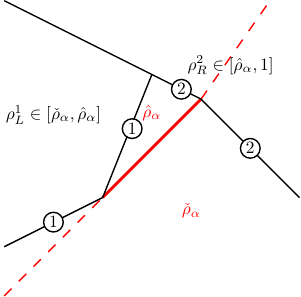}
      \centering
  \caption*{Case $b)$}
   \end{minipage}
 \caption{ \label{CN24} An example of a \textbf{\textcolor{blue}{NC2-a)}} interaction (Case a)) and of a \textbf{\textcolor{blue}{NC2-b)}} interaction (Case b)).}
\end{figure}

 \item[\textbf{\textcolor{blue}{NC2:}}]  A \textbf{\textcolor{blue}{NC2}} interaction is obtained combining Figure \ref{NSright} \textcolor{red}{a)} with Figure \ref{NSleft} \textcolor{red}{b)}. We mimic the proof of the previous case  keeping the same notations and taking in account that $\psi(\rho_L^{2}(t_2^-),\rho_R^{2}(t_2^-))$ may be different of zero (see  Figure \ref{CN24}). From Lemma \ref{dd} and Lemma \ref{ee}, the equalities in \eqref{paaa} hold and  by definition of $t_2$ $\xi_b(t_2^-)=\xi_b(t_1^+)$. 
We distinguish two different cases. By construction, there exists a couple $(k_1,k_2) \in K(n,T)^2$ such that $1\in K(n,t_1^+,T,k_1)=K(n,t_1^-,T,k_1)$ and $2 \in K(n,t_2^+,T,k_2)=K(n,t_2^-,T,k_2)$.
  \begin{itemize}
    \item[\textbf{\textcolor{blue}{NC2-a)}}] $k_1 \neq k_2$ (see Figure \ref{CN24} \textcolor{red}{a)}); roughly speaking the wave $1$ never interacts with the wave $2$. Let $k_0 \in K(n,t_2^+,T,k_1)$ such that $1\in K(n,t_1^+,t_2^+,k_0)$. From  \eqref{paaa}, we have 
\begin{equation} \label{casrare2}
 \left\{\begin{array}{l}
 \xi_b(t_2^+)=(1-\psi(\rho_L^{2}(t_2^-),\rho_R^{2}(t_2^-))) \xi_b(t_1^-)+\frac{\psi(\rho_L^{2}(t_2^-),\rho_R^{2}(t_2^-))}{\Delta \rho_2(t_2^-)}\Delta \rho_2(t_2^-)\xi_2(t_2^-),\\
\Delta \rho_2(t_2^+) \xi_{2}(t_2^+)=\sum_{k\in K(n,t_1^-,t_2^+,2)}\Delta \rho_k \xi_k+(\check{\rho}_\alpha-\hat{\rho}_\alpha)\xi_{b}(t_1^-),\\
\Delta \rho_{k_0}(t_2^+) \xi_{k_0}(t_2^+)=\sum_{k\in K(n,t_1^-,t_2^+,k_0)}\Delta \rho_k \xi_k-(\check{\rho}_\alpha-\hat{\rho}_\alpha)\xi_{b}(t_1^-).\\
 \end{array}\right.
 \end{equation}

\begin{lemma} \label{rare2} A wave coming from the left  cannot interact with the SV trajectory at time $t>t_2$. In particular, a \textbf{\textcolor{blue}{NC2-a)}} interaction occurs at most one time.
  \end{lemma}
   \begin{proof}
  The proof is obtained by mimicking the proof of Lemma \ref{rare}.
  \end{proof}
 \begin{remark}  \label{rare3}
 From Lemma \ref{rare}, if a  \textcolor{blue}{NC1-a)} interaction occurs over $[t_1,t_2]$, a \textcolor{blue}{NC2-a)} interaction can not happen on $[t_2,T]$. Reciprocally, using Lemma \ref{rare2}, if a  \textcolor{blue}{NC2-a)} interaction occurs over $[t_1,t_2]$, a \textcolor{blue}{NC1-a)} interaction can not happen on $[t_2,T]$. Thus, either  \textcolor{blue}{NC1-a)} interaction or  \textcolor{blue}{NC2-a)} interaction can occur but not both. 
 \end{remark}

 \item[\textbf{\textcolor{blue}{NC2-b)}}] $k_1=k_2$ (see Figure \ref{CN24} \textcolor{red}{b)}); roughly speaking the wave $1$  interacts with the wave $2$. Since $t_2$ is the first time after $t_1$ where a wave interacts with the SV trajectory and $k_1=k_2$, then there exist $t_3 \in(t_2,T]$ and a wave $3 \in K(n,T,t_3^+,k_1)$ such that every wave $k \in K(n,t_1^+,t_3^+,3)\backslash \{2\}$ does not interact with the SV trajectory. Thus, we have \begin{equation} \label{aaa2}\Delta \rho_3(t_3^+) \xi_3(t_3^+)=\Delta \rho_1(t_1^+) \xi_1(t_1^+)+\Delta \rho_2(t_2^+) \xi_2(t_2^+) + \sum_{k\in K(n,t_1^+,t_3^+,3)\backslash \left(K(n,t_1^+,t_2^+,2)\cup\{1\}\right)} \Delta \rho_k \xi_k. \end{equation}
Using \eqref{paaa}, \eqref{aaa2} and the equality $\Delta \rho_2(t_2^-) \xi_2(t_2^-)=\sum_{k\in K(n,t_1^-,t_2^-,2)}\Delta \rho_k \xi_k$, 
 \begin{equation} \label{panda2228}
 \left\{\begin{array}{l}
 \xi_b(t_2^+)=  (1-\psi(\rho_L^{2}(t_2^-),\rho_R^{2}(t_2^-))) \xi_b(t_1^-)+\frac{\psi(\rho_L^{2}(t_2^-),\rho_R^{2}(t_2^-))}{\Delta \rho_2(t_2^-)}\left(\sum_{k\in K(n,t_1^-,t_2^-,2)}\Delta \rho_k \xi_k\right),\\\\
 \Delta \rho_3(t_3^+) \xi_{3}(t_3^+)=\sum_{k\in K(n,t_1^-,t_3^+,3)} \Delta \rho_k \xi_k.
 \end{array}\right.
 \end{equation}


 \end{itemize}



 \begin{figure}
   \begin{minipage}[c]{.5\linewidth}
   \centering
      \includegraphics[width=0.6\linewidth]{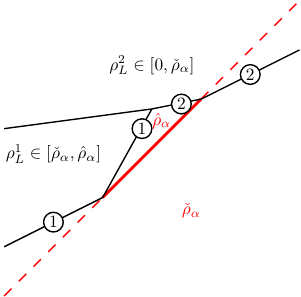}
      \caption*{Case $a)$}  
   \end{minipage} \hfill
   \begin{minipage}[c]{.5\linewidth}
      \includegraphics[width=0.6\linewidth]{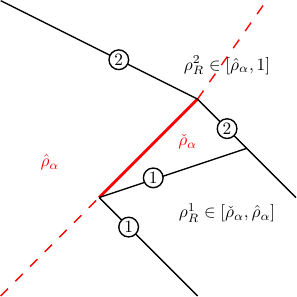}
      \centering
  \caption*{Case $b)$}
   \end{minipage}
 \caption{ \label{CN2} An example of a \textbf{\textcolor{blue}{NC3)}} interaction (Case a)) and of a \textbf{\textcolor{blue}{NC4)}} interaction (Case b)) }
\end{figure}

  
\item[\textbf{\textcolor{blue}{NC3:}}] A \textbf{\textcolor{blue}{NC3}} interaction is obtained combining Figure \ref{NSleft} \textcolor{red}{a)} with Figure \ref{NSleft} \textcolor{red}{b)}. The wave $1$ $(\rho_L^{1}(t_1^-),\rho_R^{1}(t_1^-))$ interacts with the SV trajectory at time $t_1$ creating a non classical shock and a wave $(\rho_L^{1}(t_1^+),\rho_R^{1}(t_1^+))$ which will be  denoted by $1$ as well. The wave $2$ $(\rho_L^{2}(t_2^-),\rho_R^{2}(t_2^-))$ interacts with the SV trajectory at time $t_2$ cancelling a non classical shock and producing a wave $(\rho_L^{2}(t_2^+),\rho_R^{2}(t_2^+))$ which will be called $2$ as well (see Figure \ref{CN2} \textcolor{red}{a)}). In this case, 
$  \rho_L^{1}(t_1^-)=\rho_L^{1}(t_1^+)\in [\check \rho_{\alpha},\hat \rho_{\alpha}]$, $\rho_R^{1}(t_1^-)=\check \rho_{\alpha}$, $\rho_R^{1}(t_1^+)= \hat \rho_{\alpha}$ and $\rho_L^{2}(t_2^-)=\rho_L^{2}(t_2^+)\in [0,\check \rho_{\alpha}]$, $\rho_R^{2}(t_2^-)=\hat \rho_{\alpha}$, $\rho_R^{2}(t_2^+)=\check \rho_{\alpha}$. Since the speed of the SV is not modified ($\psi(\rho_L^{2}(t_2^-),\rho_R^{2}(t_2^-))=0$), from Lemma \ref{dd} and Lemma \ref{ee}, we get
 \begin{equation} \label{panda2213}
 \left\{\begin{array}{l}
 \xi_b(t_2^+)= \xi_b(t_1^-),\\
\Delta \rho_2(t_2^+) \xi_{2}(t_2^+)=\Delta \rho_2(t_2^-) \xi_{2}(t_2^-)+(\check{\rho}_\alpha-\hat{\rho}_\alpha)\xi_{b}(t_1^-),\\
\Delta \rho_1(t_1^+) \xi_{1}(t_1^+)+(\check{\rho}_\alpha-\hat{\rho}_\alpha)\xi_{b}(t_1^-)=\Delta \rho_1(t_1^-) \xi_{1}(t_1^-).\\
 \end{array}\right.
 \end{equation}
By construction, there exists $k_1 \in K(n,T)$ such that $\{1\} \in K(n,t_1^+,T,k_1)$ and  $\{2\} \in K(n,t_2^+,T,k_1)$. Thus, we have 
 \begin{equation} \label{aaa13}\Delta \rho_{2}(t_2^-) \xi_{2}(t_2^-)=\Delta \rho_{1}(t_1^+) \xi_{1}(t_1^+)+\sum_{k\in K(n,t_1^+,t_2^-,2)\backslash \{1\}} \Delta \rho_k \xi_k. \end{equation}
Using \eqref{panda2213}, \eqref{aaa13}, we conclude that 
 \begin{equation} \label{panda2227}
 \left\{\begin{array}{l}
 \xi_b(t_2^+)= \xi_b(t_1^-),\\
 \Delta \rho_2(t_2^+) \xi_{2}(t_2^+)=\sum_{k\in K(n,t_1^-,t_2^+,2)} \Delta \rho_k \xi_k.
 \end{array}\right.
 \end{equation}


 \item[\textbf{\textcolor{blue}{NC4:}}] A \textbf{\textcolor{blue}{NC4}} interaction is obtained combining Figure \ref{NSright} \textcolor{red}{a)} with Figure \ref{NSright} \textcolor{red}{b)} (see Figure \ref{CN2} \textcolor{red}{b)}).
 In this case, 
$  \rho_R^{1}(t_1^-)=\rho_R^{1}(t_1^+)\in [\check \rho_{\alpha},\hat \rho_{\alpha}]$, $\rho_L^{1}(t_1^-)=\hat \rho_{\alpha}$, $\rho_L^{1}(t_1^+)= \check \rho_{\alpha}$ and $\rho_R^{2}(t_2^-)=\rho_R^{2}(t_2^+)\in [\hat \rho_{\alpha},1]$, $\rho_L^{2}(t_2^-)=\check \rho_{\alpha}$, $\rho_R^{2}(t_2^+)=\hat \rho_{\alpha}$.  Moreover, $\psi(\rho_L^{2}(t_2^-),\rho_R^{2}(t_2^-))$ may be different of zero. From Lemma \ref{dd} and Lemma \ref{ee},  the equalities in \eqref{paaa} hold and,  by definition of $t_2$, $\xi_b(t_2^-)=\xi_b(t_1^+)$. By construction, there exists $k_1 \in K(n,T)$ such that $\{1\} \in K(n,t_1^+,T,k_1)$ and  $\{2\} \in K(n,t_2^+,T,k_1)$. Thus, we have  \begin{equation} \label{aaa00}\Delta \rho_{2}(t_2^-) \xi_{2}(t_2^-)=\Delta \rho_{1}(t_1^+) \xi_{1}(t_1^+)+\sum_{k\in K(n,t_1^+,t_2^-,2)\backslash \{1\}} \Delta \rho_k \xi_k. \end{equation}
 Using \eqref{paaa} and \eqref{aaa00}, we conclude that
 \begin{equation} \label{panda2226}
 \left\{\begin{array}{l}
 \xi_b(t_2^+)=  \left(1-\psi(\rho_L^{2}(t_2^-),\rho_R^{2}(t_2^-))+\frac{\psi(\rho_L^{2}(t_2^-),\rho_R^{2}(t_2^-))(\hat{\rho}_\alpha-\check{\rho}_\alpha)}{\Delta \rho_2(t_2^-)}\right) \xi_b(t_1^-), \\ \quad \quad\quad+\frac{\psi(\rho_L^{2}(t_2^-),\rho_R^{2}(t_2^-))}{\Delta \rho_2(t_2^-)}\left(\sum_{k\in K(n,t_1^-,t_2^-,2)}\Delta \rho_k \xi_k\right),\\
 \Delta \rho_2(t_2^+) \xi_{2}(t_2^+)=\sum_{k\in K(n,t_1^-,t_2^+,2)} \Delta \rho_k \xi_k.
 \end{array}\right.
 \end{equation}

\begin{lemma}\label{finite}
Let $t\in [0,t_1^-)$, a wave $k$ with $k\in K(n,t,t_1^-,1)$ can not interact with the SV trajectory  at time $t$.
\end{lemma}
\begin{proof}
By contradiction, assume that some waves $k$ with $k \prec_n 1$ interact with the SV trajectory on $(0,t_1^-)$. Let $k$ be such that the time $t_0$ is the maximal interaction time on $(0,t_1^-)$ with the SV trajectory. Then, 	necessarily, the interaction is from the left and it is of type Figure  \ref{NSleft} \textcolor{red}{a)} and Figure  \ref{left1} (Figure  \ref{NSleft} \textcolor{red}{b)}   being excluded). Thus, $\rho_L^k(t_0^-)=\rho_L^k(t_0^+) \in [0,\check \rho_{\alpha}]$. Moreover, no wave can interact from the left with $k$ or any $k'$ verifying $k \prec_n k'$ on $(t_0,t_1^-)$, otherwise $t_0$ would not be maximal. This implies that $k' \succ_n k$ and for $t\in (t_0^+,t_1^-)$, $\rho_L^{k'}(t)=\rho_L^k(t_0^+) \in [0,\check \rho_{\alpha}]$ but this contradicts with $\rho_L^1(t_1^-)=\hat \rho_{\alpha}$.

\end{proof}



\end{itemize}
\textbf{Conclusion.} 
\begin{itemize}
\item Using Lemma \ref{rare}, Lemma \ref{rare2} and Remark \ref{rare3}, we conclude that  either  a \textbf{\textcolor{blue}{NC1-a)}} interaction or a \textbf{\textcolor{blue}{NC2-a)}} interaction can occur but not both. Morever they  happen at most one time and the evolution of tangent vectors is described in \eqref{casrare2}. 
\item From \eqref{panda2229} and \eqref{panda2227}, a  \textbf{\textcolor{blue}{NC1-b)}} interaction and a \textbf{\textcolor{blue}{NC3}} interaction have the same effect as  multiple wave-wave interactions (where the evolution of tangent vectors is described in Lemma \ref{2W}).
\item Using \eqref{panda2228} and $\xi_b(t_1^-)=\xi_b(t_2^-)$, a  \textbf{\textcolor{blue}{NC2-b)}} interaction has the same effect  as a classical wave-SV interaction at time $t_2$ (where the evolution of tangents vectors is described in Lemma \ref{cc}) with multiple wave-wave interactions (where the evolution of tangent vectors is described in Lemma \ref{2W}). 
\item Notice that we may have many interactions of type \textbf{\textcolor{blue}{NC4)}}. But, thanks to Lemma \ref{finite} the involved waves do not interact in the following sense: if $k_1$ and $k_2$ are involved in the creation of two different  \textbf{\textcolor{blue}{NC4)}} types then $k_1 \prec_n k_2$ and $k_2 \prec_n k_1$ can not hold true (the evolution of tangent vectors is described in \eqref{panda2226}).
\end{itemize}

\subsubsection{Proof of Lemma \ref{ffff}} \label{sectionsection}

\begin{itemize}
\item[A)] If a non classical wave is created  by a wave $k$ at time $t=t_1$ and it is cancelled at time $t_2>T$, then, using Lemma \ref{dd},
$$ \vert \Delta \rho_k(T) \xi_{k}(T)\vert+\vert \xi_b(T) \vert   \leq \vert \Delta \rho_k(t_1) \xi_{k}(t_1)\vert +(1+\hat{\rho}_\alpha-\check{\rho}_\alpha)\vert \xi_{b}(t_1)\vert.$$
\item[B)] If a non classical wave in the initial datum is cancelled  by a wave $k$ at time $t=t_2$ then, using Lemma \ref{dd},
$$ \vert \Delta \rho_k(t_2) \xi_{k}(t_2)\vert+\vert \xi_b(t_2) \vert   \leq \vert \Delta \rho_k(0) \xi_{k}(0)\vert +(1+\hat{\rho}_\alpha-\check{\rho}_\alpha)\vert \xi_{b}(0)\vert.$$
\end{itemize}
\begin{itemize}
 \item The cases A) and B) may happen at most one time and  $(1+\check{\rho}_\alpha-\hat{\rho}_\alpha)<\infty$. Thus, the weight functions $W_k^n(0)$ and $W_b^n(0)$ defined in Lemma \ref{ffff}  can not blow up because of   interactions $A) $ or   $B)$.  
 \item From Lemma \ref{rare} and Lemma \ref{rare2}, either  a \textbf{\textcolor{blue}{NC1-a)}} interaction or a \textbf{\textcolor{blue}{NC2-a)}}  interaction may occur at most one time. Moreover, $\hat{\rho}_\alpha-\check{\rho}_\alpha<\infty$ and from Lemma \ref{bababa}, \\ 
 $$\displaystyle{(1-\psi(\rho_L^{2}(t_2^-),\rho_R^{2}(t_2^-)))<\infty} \quad \text{and} \quad\frac{\psi(\rho_L^{2}(t_2^-),\rho_R^{2}(t_2^-))}{\Delta \rho_2(t_2^-)}<\infty.$$ Thus, the weight function $W_k^n(0)$ and $W_b^n(0)$ defined in Lemma \ref{ffff}  can not blow up because of  \textbf{\textcolor{blue}{NC1-a)}} interactions or  \textbf{\textcolor{blue}{NC2-a)}} interactions. 
 \item  \textbf{\textcolor{blue}{NC1-b)}} interactions and \textbf{\textcolor{blue}{NC3}} interactions produce tangent vector increases as   multiple wave-wave interactions. 
 \item  \textbf{\textcolor{blue}{NC2-b)}} interactions have the same effect as a classical wave-SV interaction at time $t_2$ (where the evolution of tangents vectors is described in Lemma \ref{cc}) with multiple wave-wave interactions (where the evolution of tangents vectors is described in Lemma \ref{2W}).
 \item The  evolution of the SV tangent vector for a \textbf{\textcolor{blue}{NC4}} interaction   is described in \eqref{panda2226}. In \eqref{panda2226}, $\Delta \rho_2(t_2^-)= \rho_R^{2}(t_2^-)-\rho_L^{2}(t_2^-) \leq \hat \rho_{\alpha}- \check \rho_{\alpha}$. Thus,  $$\displaystyle{\left(1-\psi(\rho_L^{2}(t_2^-),\rho_R^{2}(t_2^-))+\frac{\psi(\rho_L^{2}(t_2^-),\rho_R^{2}(t_2^-))(\hat{\rho}_\alpha-\check{\rho}_\alpha)}{\Delta \rho_2(t_2^-)}\right)\leq 1}.$$ We conclude that $\xi_b(t_2^+) \leq \xi_b(t_1^-)$, that is to say a \textbf{\textcolor{blue}{NC4}} interaction does not increase the value of the SV tangent vector $\xi_b$ over time.  From Lemma \ref{finite}, Lemma \ref{bababa} and \eqref{panda2226}
 $$ \begin{array}{ll} \Delta \rho_2(t_2^+) \xi_{2}(t_2^+) &= \frac{\psi(\rho_L^{2}(t_2^-),\rho_R^{2}(t_2^-))}{\Delta \rho_2(t_2^-)}\left(\sum_{k\in K(n,0,t_2^-,2)}\Delta \rho_k(0) \xi_k(0)\right) \\
 &<\frac{2}{\rho^*}\left(\sum_{k\in K(n,0,t_2^-,2)}\Delta \rho_k(0) \xi_k(0)\right),
 \end{array}
 $$
 where the notation are described in Lemma \ref{finite}. Thus, the weight function $W_k^n(0)$ and $W_b^n(0)$ defined in Lemma \ref{ffff}  can not blow up because of  \textbf{\textcolor{blue}{NC4}} interactions.
  \end{itemize}
  Our goal is to prove that the weight functions $W_k^n(0)$ and $W_b^n(0)$, defined in Lemma \ref{ffff} can not blow up when $n$ tends to infinity.  Thus, we need only consider the interactions which may blow  the weight functions $W_k^n(0)$ and $W_b^n(0)$ up. More precisely, we only take into account the simple wave-SV interactions, the wave-wave interactions, the  \textbf{\textcolor{blue}{NC1-b)}} interactions, \textbf{\textcolor{blue}{NC2-b)}} interactions and the \textbf{\textcolor{blue}{NC3}} interactions.
\begin{lemma} \label{coco} Let $t_1,t_2 \in \R_+^*$ such that $t_1<t_2$. We assume that only the following interactions may occur
\begin{itemize}
\item  multiple wave-wave interactions (the evolution of tangent vectors is described in Lemma \ref{bb}, Figure \ref{2W}),
\item  wave-SV interactions (the evolution of tangent vectors is described in Lemma \ref{cc}, Figure \ref{modbus} and Figure \ref{left1}),
\item the \textbf{\textcolor{blue}{NC1-b)}} interactions (the evolution of tangent vectors is described in \eqref{panda2229}, Figure \ref{CN3} \textcolor{red}{b)})
\item the  \textbf{\textcolor{blue}{NC2-b)}} interactions (the evolution of tangent vectors is described in \eqref{panda2228}, Figure \ref{CN24} \textcolor{red}{b)}),
\item  the  \textbf{\textcolor{blue}{NC3}} interactions  (the evolution of tangent vectors is described in \eqref{panda2227}, Figure \ref{CN2} \textcolor{red}{a)}),
\end{itemize}
 Then there exist $W_b(t_1,t_2) \in \R_+^*$ and  $W_k(t_1,t_2)\in  \R_+^*$  for $k\in K(n,t_1)$ such that 
\begin{equation} \label{ffffefe}
\left\{\begin{array}{l}
\xi_b(t_2)=W_b(t_1,t_2) \xi_b(t_1)+ \sum_{k \in K(n,t_1)} W_k(t_1,t_2) \Delta  \rho_k \xi_k,\\
\Delta \rho_j(t_2) \xi_j(t_2)=\sum_{k \in K(n,t_1,t_2,j)} \Delta \rho_k \xi_k, \quad \text{for every} \, \, j\in K(n,t_2).  \\
\end{array}\right.
\end{equation}
Moreover, we have
 \begin{equation} \label{vital} \max(\vert W_b(t_1,t_2) \vert ,\vert W_k(t_1,t_2)\vert)<\frac{2\exp\left(\frac{3TV(\rho_0)}{\rho^*}\right)}{\rho^*}. \end{equation}
\end{lemma}
\begin{figure}
  \centering
      \includegraphics[width=0.5\linewidth]{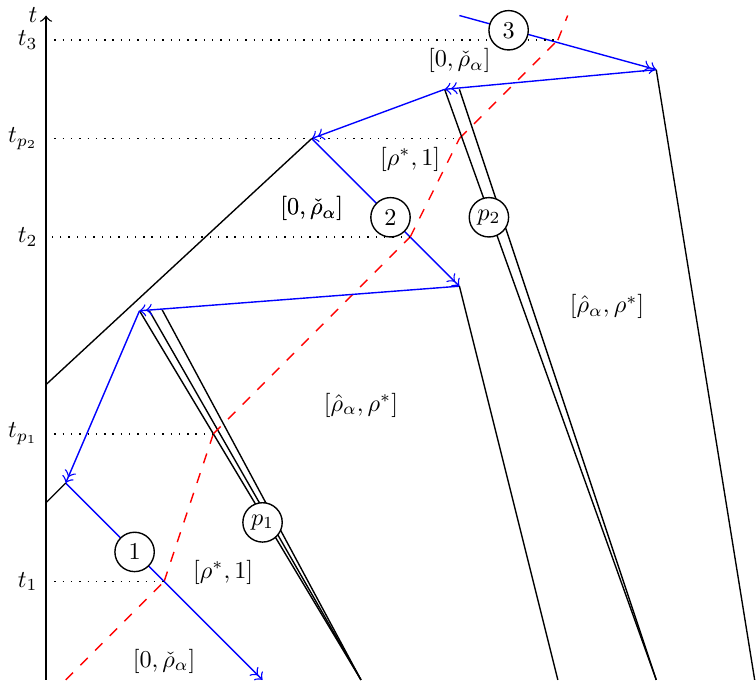}
   \caption{ \label{2cl}  Interactions with  \textcolor{red}{the SV trajectory} (\textcolor{red}{$\cdots$}). }
\end{figure}
The proof is postponed in Appendix \ref{explicit}. The  following computations of the evolution of tangent vectors  in the case of Figure \ref{2cl} highlights the main ideas of the proof of Lemma  \ref{coco}.  \\ 
 \ \\
\textbf{Example:} We consider the particular case presented in Figure \ref{2cl}.  To simplify the notations,  $\psi(\rho_L^{i}(t_i^-),\rho_R^{i}(t_i^-))$ will be denoted by $\psi^i$.
From Lemma \ref{cc},
\begin{equation} \label{qwdqwd22}
\left\{ \begin{array}{l}\xi_b(t_3^+)=(1-\psi^3)\xi_b(t_3^-)+\frac{\psi^3}{\Delta \rho_3(t_3^-)}\Delta \rho_3(t_3^-)\xi_{3}(t_3^-),\\
\Delta \rho_3(t_3^+) \xi_2(t_3^+)=\Delta \rho_3(t_3^-) \xi_3(t_3^-), \\ 
\xi_b(t_{p_2}^+)=(1-\psi^{p_2})\xi_b(t_{p_2}^-)+\frac{\psi^{p_2}}{\Delta \rho_{p_2}(t_{p_2}^-)}\Delta \rho_{p_2}(t_{p_2}^-) \xi_{1}(t_{p_2}^-),\\
\Delta \rho_{p_2}(t_{p_2}^+) \xi_{p_2}(t_{p_2}^+)=\Delta \rho_{p_2}(t_{p_2}^-) \xi_{p_2}(t_{p_2}^-). \\ \end{array} \right.
\end{equation}
Since $\xi_b(t_3^-)=\xi_b(t_{p_2}^+)$ and $\Delta \rho_3(t_3^-) \xi_3(t_3^-)=\Delta \rho_{p_2}(t_{p_2}^+) \xi_{p_2}(t_{p_2}^+)+\sum_{k\in K(n,t_{p_2}^+,t_3^-,3)\backslash \{p_2^+\}} \Delta \rho_k \xi_k$, we have 
\begin{equation} \label{qwdqwd33}
\left\{ \begin{array}{l}\xi_b(t_3^+)=(1-\psi^3)(1-\psi^{p_2})\xi_b(t_{p_2}^-)+ 
\left(\frac{\psi^3}{\Delta \rho_3(t_3^-)}+(1-\psi^3)\frac{\psi^{p_2}}{\Delta \rho_{p_2}(t_{p_2}^-)}
\right)\Delta \rho_{p_2}(t_{p_2}^-)\xi_{{p_2}}(t_{p_2}^-)\\
\quad \quad \quad+ \frac{\psi^3}{\Delta \rho_{3}(t_{3}^-)}\sum_{k\in K(n,t_{p_2}^-,t_3^+,3)\backslash \{p_2^+\}} \Delta \rho_k \xi_k,\\
\Delta \rho_{3}(t_{3}^+) \xi_{3}(t_{3}^+)=\sum_{k\in K(n,t_{p_2}^-,t_3^+,3)} \Delta \rho_k \xi_k.
\end{array} \right.
\end{equation}
 From Lemma \ref{cc},
\begin{equation} \label{qwdqwd44}
\left\{ \begin{array}{l}\xi_b(t_2^+)=(1-\psi^2)\xi_b(t_2^-)+\frac{\psi^2}{\Delta \rho_2(t_2^-)}\Delta \rho_2(t_2^-)\xi_{2}(t_2^-),\\
\Delta \rho_2(t_2^+) \xi_2(t_2^+)=\Delta \rho_2(t_2^-) \xi_2(t_2^-). \\ 
  \end{array} \right.
\end{equation}

Using $\xi_b({p_2}^-)=\xi_b(t_2^+)$, $\Delta \rho_{p_2}(t_{p_2}^-) \xi_{p_2}(t_{p_2}^-)=\Delta \rho_{2}(t_{2}^+) \xi_{2}(t_{2}^+)+\sum_{k\in K(n,t_{2}^+,t_{p_2}^-,{p_2})\backslash \{2^+\}} \Delta \rho_k \xi_k$, \eqref{qwdqwd33} and \eqref{qwdqwd44}, we have \small
\begin{equation*} 
\left\{ \begin{array}{l}\xi_b(t_3^+)=(1-\psi^3)(1-\psi^{p_2})(1-\psi^{2})\xi_b(t_{2}^-)+
\left(\frac{\psi^3}{\Delta \rho_3(t_3^-)}+ (1-\psi^3)\frac{\psi^{p_2}}{\Delta \rho_{p_2}(t_{p_2}^-)}
\right)\Delta \rho_{p_2}(t_{2}^-)\xi_{{p_2}}(t_{2}^-)\\
\quad \quad+\left(\frac{\psi^3}{\Delta \rho_3(t_3^-)}+ (1-\psi^3)(1-\psi^{p_2})\frac{\psi^{2}}{\Delta \rho_{2}(t_{2}^-)}
\right)\Delta \rho_{2}(t_{2}^-)\xi_{{2}}(t_{2}^-)+ \frac{\psi^{p_2}}{\Delta \rho_{p_2}(t_{p_2}^-)}\sum_{k\in K(n,t_{2}^-,t_3^+,3)\backslash \{p_2^+,2\}} \Delta \rho_k \xi_k,\\
\Delta \rho_{3}(t_{3}^+) \xi_{3}(t_{3}^+)=\sum_{k\in K(n,t_{2}^-,t_3^+,3)} \Delta \rho_k \xi_k.
\end{array} \right.
\end{equation*}
\normalsize
From Lemma \ref{cc},
\begin{equation} \label{qwdqwd77}
\left\{ \begin{array}{l}\xi_b(t_{p_1}^+)=(1-\psi^{p_1})\xi_b(t_{p_1}^-)+\frac{\psi^{p_1}}{\Delta \rho_{p_1}(t_{p_1}^-)}\Delta \rho_{p_1}(t_{p_1}^-)\xi_{p_1}(t_{p_1}^-),\\
\Delta \rho_{p_1}(t_{p_1}^+) \xi_{p_1}(t_{p_1}^+)=\Delta \rho_{p_1}(t_{p_1}^-) \xi_{p_1}(t_{p_1}^-), \\ 
\xi_b(t_{1}^+)=(1-\psi^{1})\xi_b(t_{1}^-)+\frac{\psi^{1}}{\Delta \rho_{1}(t_{1}^-)}\Delta \rho_{1}(t_{1}^-)\xi_{1}(t_{1}^-),\\
\Delta \rho_{1}(t_{1}^+) \xi_{1}(t_{1}^+)=\Delta \rho_{1}(t_{1}^-) \xi_{1}(t_{1}^-). \\ 
  \end{array} \right.
\end{equation}
By straightforward computations, we conclude that

\begin{equation*}
\left\{ \begin{array}{l}\xi_b(t_3^+)=W_b^n(0)\xi_b(t_{1}^-)+\sum_{k\in K(n,t_1^-,3,t_3^+)\backslash \left(\{p_2\} \cup K(n,t_1^-,2,t_2^+) \right)}W_k^n(0)\Delta \rho_k \xi_k+  W^n_{p_2}(0) \rho_{p_2}(0)\xi_{{p_2}}(0)\\+\sum_{k\in K(n,t_1^-,2,t_2^+)\backslash \{p_1 \cup 1\}}W_k^n(0)\Delta \rho_k \xi_k+W^n_{p_1}(0) \Delta\rho_{p_1}(0)\xi_{{p_1}}(0)+W_1^n(0)\Delta \rho_1(0)\xi_{{1}}(0),\\
\Delta \rho_{3}(t_{3}^+) \xi_{3}(t_{3}^+)=\sum_{k\in K(n,t_{1}^-,t_3^+,3)} \Delta \rho_k \xi_k,
\end{array} \right.
\end{equation*}
 with \begin{equation}\label{lolo}\left\{ \begin{array}{l}
W_b^n(0)=(1-\psi^3)(1-\psi^{p_2})(1-\psi^{2})(1-\psi^{p_1})(1-\psi^{1}),\\
W_k^n(0)=\frac{\psi^3}{\Delta \rho_3(t_3^-)}, \quad\forall k\in K(n,t_1^-,3,t_3^+)\backslash \left(\{p_2\} \cup K(n,t_1^-,2,t_2^+) \right),\\
W_{p_2}^n(0)=\frac{\psi^3}{\Delta \rho_3(t_3^-)}+ (1-\psi^3)\frac{\psi^{p_2}}{\Delta \rho_{p_2}(t_{p_2}^-)},\\
W_k^n(0)=\frac{\psi^3}{\Delta \rho_3(t_3^-)}+ (1-\psi^3)(1-\psi^{p_2})\frac{\psi^{2}}{\Delta \rho_{2}(t_{2}^-)}, \quad \forall k\in K(n,t_1^-,2,t_2^+)\backslash \{p_1 \cup 1\},\\
W_{p_1}^n(0)=\frac{\psi^3}{\Delta \rho_3(t_3^-)}+ (1-\psi^3) (1-\psi^{p_2}) \left(\frac{\psi^2}{\Delta \rho_2(t_2^-)}+(1-\psi^2)\frac{\psi^{p_1}}{\Delta \rho_{p_1}(t_{p_1}^-)}\right),\\
W_1^n(0)=\frac{\psi^3}{\Delta \rho_3(t_3^-)}+ (1-\psi^3) (1-\psi^{p_2}) \left(\frac{\psi^2}{\Delta \rho_2(t_2^-)}+(1-\psi^2)(1-\psi^{p_1})\frac{\psi^{1}}{\Delta \rho_{1}(t_{1}^-)}\right).\\
 \end{array}\right.
 \end{equation}
Thus we have, 
$$\vert W_b^n \vert \leq \left(1+\frac{3}{2^{n+1}\rho^*}\right)^5 < \exp \left(\frac{15}{2\rho^*}\right).$$ 
We can notice that if a wave $k$ interacts with the SV trajectory $n$ times then $W_k^n$ will be decomposed into a sum of $n$ terms. We want to find a bound of $W_k^n$ which does not depend on $n$. For instance, we have 
$$\vert W_1^n(0)\vert \leq\left(1+\frac{3}{2^{n+1}\rho^*}\right)^2 \left(\vert \frac{\psi^3}{\Delta \rho_3(t_3^-)}\vert+ \vert (1-\psi^3)\vert \left(\vert \frac{\psi^2}{\Delta \rho_2(t_2^-)}\vert+\vert(1-\psi^2)\vert \frac{2}{\rho^*}\right)\right).$$
The two following points are the main tools to prove Lemma  \ref{coco}:
\begin{itemize}
\item The wave $1$ (resp. the wave $2$) interacts with the SV trajectory modifying the speed of the SV at time $t_1$ (resp. at time $t_2$) with $t_1<t_2$ and $1\in K(n,t_1,t_2,2)$ then $\rho_L^{2}(t_2^-) \in  (0,\check{\rho}_{\alpha})$. This statement is proved in Lemma \ref{pp}.
\item Since $\rho_L^{2}(t_2^-) \in  (0,\check{\rho}_{\alpha})$, from Lemma \ref{bababa}, we have $ 
\left(\vert \frac{\psi^2}{\Delta \rho_2(t_2^-)}\vert+\vert(1-\psi^2)\vert \frac{2}{\rho^*}\right)\leq \frac{2}{\rho^*}.$ We conclude that 
$$\vert W_1^n(0)\vert \leq \left(1+\frac{3}{2^{n+1}\rho^*}\right)^2 \left(\vert \frac{\psi^3}{\Delta \rho_3(t_3^-)}\vert+ \vert (1-\psi^3)\vert \frac{2}{\rho^*}\right).$$
Repeating the same process, we get
$$\vert W_1^n(0)\vert \leq \frac{2\exp\left( \frac{3}{\rho^*}\right)}{\rho^*},$$
which does not depend on $n$.
\end{itemize}

\textbf{Proof of Theorem \ref{main1}:} From Lemma \ref{ffff}, for every $T>0$ and $n\in N^*$
$$\sum_{k\in K(n,T)} \vert \Delta \rho^n_k(T)\xi_k^n(T)\vert+\vert \xi^n_b(T)\vert < C\sum_{k\in K(n,0)} \vert \Delta \rho^n_k(0)\xi_k^n(0)\vert+\vert \xi^n_b(0)\vert,$$
with $C>0$ a constant independent of $n$. Thus, using Section \ref{WW}, we have
$$ \Vert \rho^{2,n}(t)-\rho^{1,n}(t) \Vert_{L^1(\R)} + \vert y^{2,n}(t)-y^{1,n}(t) \vert \leq C(  \Vert \rho_0^{2,n}-\rho_0^{1,n} \Vert_{L^1(\R)} + \vert y_0^{2}-y_0^{1} \vert),$$
From \cite[Lemma 3]{DP14} we pass to the limit of the previous inequality and we conclude the proof of Theorem \ref{main1}. \cqfd

\begin{remark} In the particular case presented in Figure \ref{2cl}, $C=\frac{2\exp\left( \frac{3}{\rho^*}\right)}{\rho^*}$.
\end{remark}
\appendix
\section*{Appendix}
\section{Proof of Lemma \ref{bababa}} \label{app::bababa}
This proof is based on the equality $\rho^*=\check{\rho}_\alpha+\hat{\rho}_\alpha$.
\begin{itemize}
\item We have $$1-\psi(\rho_L,\rho_R)=\left\{ \begin{array}{cl}
\frac{\rho_l}{\rho_L+\rho_R-\rho^*} &  \text{if} \quad (\rho_R>\rho^*  \, \, \& \, \,  \rho_L\in [0,\check{\rho}_\alpha]\cup[\hat{p}_\alpha,\rho^*]),\\
\frac{\rho_L}{\rho_R} &  \text{if} \quad (\rho_R>\rho^* \, \, \& \, \, \rho_L\in [\rho^*,\rho_R]) \quad \text{or} \quad  (\rho^* \leq \rho_R <\rho_L), \\
0 & \text{otherwise}.\\
\end{array}\right.$$
Thus, if $(\rho_R>\rho^*  \, \, \& \, \,  \rho_L\in [0,\check{\rho}_\alpha]\cup[\hat{\rho}_\alpha,\rho^*])$ or $(\rho_R>\rho^* \, \, \& \, \, \rho_L\in [\rho^*,\rho_R])$ we have $\vert 1-\psi(\rho_L,\rho_R)\vert \leq 1$ and if $(\rho^* \leq \rho_R <\rho_L)$, using $\rho_L-\rho_R\leq \frac{3}{2^{n+1}}$, we conclude that 
$ \vert 1-\psi(\rho_L,\rho_R)\vert \leq 1+\frac{3}{2^{n+1}\rho^*}.$
\item We get $$\frac{\psi(\rho_L,\rho_R)}{\rho_R-\rho_L}=\left\{ \begin{array}{cl}
\frac{\rho_R-\rho^*}{(\rho_L+\rho_R-\rho^*)(\rho_R-\rho_L)} & \quad \text{if} \quad (\rho_R>\rho^*  \, \, \& \, \,  \rho_L\in [0,\check{\rho}_\alpha]\cup[\hat{\rho}_\alpha,\rho^*]),\\
\frac{1}{\rho_R} & \quad \text{if} \quad (\rho_R>\rho^* \, \, \& \, \, \rho_L\in [\rho^*,\rho_R]) \quad \text{or} \quad  (\rho^* \leq \rho_R <\rho_L). \\
\end{array}\right.$$
If $(\rho_R>\rho^* \, \, \& \, \, \rho_L\in [\rho^*,\rho_R]) \quad \text{or} \quad  (\rho^* \leq \rho_R <\rho_L)$ we have immediately  $\frac{1}{\rho_R}\leq \frac{2}{\rho^*}$.
If  $(\rho_R>\rho^*  \, \, \& \, \,  \rho_L\in [0,\check{\rho}_\alpha])$ then 
$\vert \frac{\rho_R-\rho^*}{(\rho_L+\rho_R-\rho^*)(\rho_R-\rho_L)} \vert \leq \frac{1}{\rho^*-\check{\rho}_{\alpha}} \leq \frac{2}{\rho^*}$ using $\rho^* \geq 2\check{\rho}_{\alpha}$. If $(\rho_R>\rho^*  \, \, \& \, \,  \rho_L\in [\hat{\rho}_\alpha,\rho^*])$ then 
$\vert \frac{\rho_R-\rho^*}{(\rho_L+\rho_R-\rho^*)(\rho_R-\rho_L)} \vert \leq \frac{1}{(\rho_L+\rho_R-\rho^*)} \leq \frac{1}{\hat{\rho}_{\alpha}} \leq \frac{2}{\rho^*}$ using $2\hat{\rho}_{\alpha}\geq \rho^*$.
\item If $\rho_R>\rho^*  \, \, \& \, \,  \rho_L\in [0,\check{\rho}_\alpha]$, then $\frac{1}{\rho_R-\rho_L} \leq \frac{1}{\rho^*-\check{\rho}_\alpha} \leq \frac{2}{\rho^*} $ using $2\check{\rho}_\alpha  \leq \rho^*$. Thus 
 $\vert \frac{\psi(\rho_L,\rho_R)}{\rho_R-\rho_L} \vert +\vert (1-\psi(\rho_L,\rho_R))\vert\frac{2}{\rho^*}= \frac{\psi(\rho_L,\rho_R)}{\rho_R-\rho_L} +(1-\psi(\rho_L,\rho_R))\frac{2}{\rho^*} \leq  \frac{2}{\rho^*}(1-\psi(\rho_L,\rho_R)+\psi(\rho_L,\rho_R))=\frac{2}{\rho^*}$.
 
\end{itemize}

\section{Proof of Lemma \ref{coco}} \label{explicit}
 From Lemma \ref{bb}, Lemma \ref{cc}, Section \ref{111} and Section \ref{explication}, \eqref{ffffefe} holds true. The proof of  the estimate \eqref{vital} is based on the following lemma
\begin{lemma} \label{pp} We only consider the interactions are described in Lemma \ref{coco}. We assume the wave $1$ (resp. the wave $2$) interacts with the SV  modifying the speed of the SV at time $t_1$ (resp. at time $t_2$) with $t_1<t_2$ and $1\in K(n,t_1,t_2,2)$ then $\rho_L^{2}(t_2^-) \in  (0,\check{\rho}_{\alpha})$.
\end{lemma}
\begin{proof} A wave modifies the speed of the SV only if a wave comes from the right of the SV trajectory (see Figure \ref{modbus} \textcolor{red}{a)}, Figure  \ref{modbus} \textcolor{red}{b)} and Figure \ref{NSright} \textcolor{red}{a)}). Using $1\in K(n,t_1,t_2,2)$ and Lemma \ref{entropy}, some waves $k$ with $1 \prec_n k$  and $k \prec_n 2$ interact with the SV on $(t_1,t_2)$. Let $k$ be such that the time $t_0$ is the maximal interaction time on $(t_1,t_2)$. Then, 	necessarily, the interaction is from the left and of type Figure  \ref{NSleft} \textcolor{red}{a)} and Figure  \ref{left1} (Figure  \ref{NSleft} \textcolor{red}{b)}   being excluded). Thus, $\rho_L^k(t_0^-)=\rho_L^k(t_0^+) \in [0,\check \rho_{\alpha}]$. Moreover, no wave can interact from the left with $k$ or any $k'$ verifying $k \prec_n k'$ on $(t_0,t_2)$, otherwise $t_0$ would not be maximal. This implies that $k' \succ_n k$ and for $t\in (t_0^+,t_2)$, $\rho_L^{k'}(t)=\rho_L^k(t_0^+) \in [0,\check \rho_{\alpha}]$. In particular,  $\rho_L^2(t_2^-)=[0,\check \rho_{\alpha}]$.
\end{proof}


\textbf{Proof of Lemma \ref{coco}.}
Since the  \textbf{\textcolor{blue}{NC1-b)}} interactions, \textbf{\textcolor{blue}{NC2-b)}} interactions and the \textbf{\textcolor{blue}{NC3}} interactions have the same effect as mutliple wave-wave interactions and wave-SV interactions. We can restrict our study to wave-wave interactions and wave-SV interactions.
Before we deal with the general case, we will consider two particular cases.
 \ \\
 \par  \textbf{Particular case $1$:} we assume that there exist a wave $1$ and a wave $2$ interacting with the SV trajectory at $t=t_1$ and $t=t_2$ as a wave-SV interaction modifying the speed of the SV  with $1\in K(n,t_1,t_2,2)$. Let $t_2$ be the first interation time after $t_1$ such that $1 \prec_n 2$. From Lemma \ref{cc},
\begin{equation} \label{qwdqwd}
\left\{ \begin{array}{l}\xi_b(t_2^+)=(1-\psi(\rho_L^{2}(t_2^-),\rho_R^{2}(t_2^-)))\xi_b(t_2^-)+\frac{\psi(\rho_L^{2}(t_2^-),\rho_R^{2}(t_2^-))}{\Delta \rho_2(t_2^-)}\Delta \rho_2(t_2^-)\xi_{2}(t_2^-),\\
\Delta \rho_2(t_2^+) \xi_2(t_2^+)=\Delta \rho_2(t_2^-) \xi_2(t_2^-), \\ 
\xi_b(t_1^+)=(1-\psi(\rho_L^{1}(t_1^-),\rho_R^{1}(t_1^-)))\xi_b(t_1^-)+\frac{\psi(\rho_L^{1}(t_1^-),\rho_R^{1}(t_1^-))}{\Delta \rho_1(t_1^-)}\Delta \rho_1(t_1^-) \xi_{1}(t_1^-),\\
\Delta \rho_1(t_1^+) \xi_1(t_1^+)=\Delta \rho_1(t_1^-) \xi_1(t_1^-). \\ \end{array} \right.
\end{equation}
Combining Lemma \ref{2W}, $1\in K(n,t_1^-,t_2^+,2)$ with \eqref{qwdqwd}, we have $\Delta \rho_{2}(t_2^+) \xi_2(t_2^+)= \Delta \rho_1(t_1^-) \xi_1(t_1^-)+\sum_{k\in K(n,t_1^-,t_2^+,2)\backslash \{1\}} \Delta \rho_k \xi_k$. We may have multiple  $m$ wave-SV interactions coming from the right. More precisely, for $i=\{1,\cdots,m\}$ with $m\in \N^*$, we assume that the $p_i$-wave intersects the SV trajectory at time $t_{p_i}$. Since $t_2$ is the first interaction time from the right,  we have  $t_1<t_{p_1}<\cdots<t_{p_m}<t_2$ and  $1 \notin K(n,t_1^+,t_{p_i}^-,p_i)$. To simplify the notations, we will write $\psi(\rho_L^{i}(t_i^-),\rho_R^{i}(t_i^-))$ as $\psi^i$.  From Lemma \ref{cc},
\begin{equation} \label{akak} \left\{ \begin{array}{l}
\xi_b(t_{p_i}^+)=(1-\psi^{p_i})  \xi_b(t_{p_i}^-)+\frac{\psi^{p_i}}{\Delta \rho_{p_i}(t_{p_i}^-)}\Delta \rho_{p_i}(t_{p_i}^-)\xi_{p_i}(t_{p_i}^-).\\
\rho_{p_i}(t_{p_i}^+)\xi_{p_i}(t_{p_i}^+)=\rho_{p_i}(t_{p_i}^-)\xi_{p_i}(t_{p_i}^-),
\end{array} \right.
\end{equation}
Since no other interaction occurs with the SV trajectory from the right  on $(t_{p_{i-1}},t_{p_{i}})$ for $i=\{2,\cdots,m\}$, 
we get  $\xi_b(t_{p_i}^-)=\xi_b(t_{p_{i-1}}^+)$, $\xi_b(t_2^-)=\xi_b(t_{p_{m}}^+)$ and $\xi_b(t_{1}^+)=\xi_b(t_{p_{1}}^-)$. Combining \eqref{qwdqwd}  with \eqref{akak} we conclude that
\begin{equation} \label{gerg} \begin{array}{lll}
\xi_b(t_{2}^+)&=& \prod_{i=1}^{m+1} (1-\psi^{p_i}) \xi_b(t_{1}^-) +\frac{\psi^2}{\Delta \rho_2(t_{2}^-)}\Delta \rho_2(t_{2}^-)\xi_{2}(t_{2}^-)\\ &+ &\sum_{k=0}^{m} \frac{\psi^{p_{k}}}{\Delta \rho_{p_{k}}(t_{p_k}^-) } \Delta \rho_{p_{k}}(t_{p_k}^-) \xi_{p_{k}}(t_{p_k}^-) \prod_{j=0}^{m-k} (1-\psi^{p_{m-j+1}}),
\end{array}
\end{equation}
with  $p_{m+1}:=2$ and $p_0:=1$. Now $1 \notin K(n,t_1,t_2^-,p_i)$. From Lemma \ref{bb}, 
\begin{equation} \label{caca3} \Delta \rho_2(t_{2}^-)\xi_{2}(t_{2}^-)=\sum_{k \in K(n,t_1^-,t_2,2)\backslash \cup_{i=0}^m K(n,t_1,t_{p_i},p_i)}\Delta \rho_k \xi_k+ \sum_{i=0}^m \sum_{k \in K(n,t_1,t_{p_i},p_i) }\Delta \rho_k \xi_k. \end{equation}
Combining \eqref{akak},\eqref{gerg} and \eqref{caca3}, we conclude that
\begin{equation} \label{f2}
\left\{\begin{array}{l}
\xi_b(t_2^+)=W_b(t_1^-,t_2^+) \xi_b(t_1^-)+ \sum_{k \in K(n,t_1^-)} W_k(t_1^-,t_2^+) \Delta \rho_k \xi_k,\\
\Delta \rho_2(t_2^+) \xi_2(t_2^+)=\sum_{k \in K(n,t_1^-,t_2^+,2)} \Delta \rho_k \xi_k,   
\end{array}\right.
\end{equation}

with $W_b(t_1^-,t_2^+)=\prod_{i=1}^{m+1} (1-\psi^{p_i})$ and 
\begin{equation} W_k(t_1^-,t_2^+)=\left\{ \begin{array}{l}
\frac{\psi^2}{\Delta \rho_2(t_2^-)} \quad \text{if} \, \, k\in K(n,t_1^-,t_2,2)\backslash \cup_{i=0}^m K(n,t_1,t_{p_i},p_i), \\
\frac{\psi^2}{\Delta \rho_2(t_2^-)}+ (1-\psi^2)\frac{\psi^{p_{i}}}{\Delta \rho_{p_{i}}(t_{p_i}^-)} \prod_{j=1}^{m-i} (1-\psi^{p_{m-j+1}}) \quad \text{if} \, \, \left\{\begin{array}{l} k \in K(n,t_1^-,t_{p_i},p_i),  \\ i \in \{0,\cdots,m\},\end{array}\right.\\
0 \quad \text{otherwise}.
\end{array}\right.
\end{equation}
From Lemma \ref{bababa}, $ \vert \frac{\psi^{p_{i}}}{\Delta \rho_{p_{i}}} \vert \leq \frac{2}{\rho^*}$ for every $i\in \{0,\cdots,m-1\}$ and $\vert \prod_{j=1}^{m-i} (1-\psi^{p_i}) \vert \leq (1+\frac{3}{2^{n+1}\rho^*})^{m}.$ Moreover,
$$\vert \frac{\psi^2}{\Delta \rho_2}+ (1-\psi^2)\frac{\psi^{p_{i}}}{\Delta \rho_{p_{i}}} \prod_{j=1}^{m-i} (1-\psi^{p_{m-j+1}})\vert \leq (1+\frac{3}{2^{n+1}\rho^*})^{m}\left (\vert \frac{\psi^{2}}{\Delta \rho_2} \vert+ \vert (1-\psi^2\vert \frac{2}{\rho^*}\right).$$ Combining Lemma \ref{bababa} and Lemma \ref{pp}, we conclude that 
$$\vert \frac{\psi^2}{\Delta \rho_2}+ (1-\psi^2)\frac{\psi^{p_i}}{\Delta \rho_{p_{i}}} \prod_{j=1}^{m-i} (1-\psi^{p_{m-j+1}})\vert \leq    \frac{2 \left(1+\frac{3}{2^{n+1}\rho^*}\right)^{m}}{\rho^*}.$$ Since, for every $\tilde{\rho}_j^n, \tilde{\rho}_i^n \in \widetilde{\mathcal{M}_n}$, $\frac{1}{2^{n+1}}\leq \vert \tilde{\rho}_j^n-\tilde{\rho}_i^n\vert$ we have $\frac{m}{2^{n+1}}\leq TV(\rho_0)$, whence
\begin{equation} \label{lala}(1+\frac{3}{2^{n+1}\rho^*})^{m} \leq (1+\frac{3}{2^{n+1}\rho^*})^{2^{n+1}TV(\rho_0)}\leq \exp\left(\frac{3TV(\rho_0)}{\rho^*}\right).\end{equation} We conclude that 
$$\max(W_b(t_1^-,t_2^+),W_k(t_1^-,t_2^+))\leq  \frac{2 \exp\left(\frac{3TV(\rho_0)}{\rho^*}\right)}{\rho^*}.$$


  \par  \textbf{Particular case $2$:} we assume that there exist $m$ wave denoted by $1,\cdots,m$ interacting with the SV trajectory at $t=t_1, t=t_2, \cdots, t=t_m$ respectively as a wave-SV interaction modifying the speed of the SV  with $1\prec_n 2 \prec_n \cdots \prec_n \prec_n m$ and no other interactions with the SV trajectory occur. That is to say  for $i=\{1,\cdots,m-1\}$, $i\in K(n,t_i,t_{i+1},i+1)$ and $\xi_b(t_{i-1}^+)=\xi_b(t_{i}^-)$.
  Substituting $p_i$ by $i$ in \eqref{gerg},  
we conclude that
 \begin{equation} \label{gerg2}
\xi_b(t_m^+)= \prod_{i=1}^m (1-\psi^{i})\xi_b(t_1^-) +\frac{\psi^{m}}{\Delta \rho_m}\Delta \rho_m\xi_{m}+\sum_{k=1}^{m-1} \frac{\psi^{k}}{\Delta \rho_{k}} \Delta \rho_{k} \xi_{k}\prod_{j=0}^{m-k-1} (1-\psi^{m-j}).
\end{equation}
  Since for $i=\{1,\cdots,m-1\}$, $i\in K(n,t_i,t_{i+1},i+1)$ we have, for every $k\in \{1,\cdots,m\}$,
 \begin{equation} \begin{array}{lll} \label{2222}
\Delta \rho_{k} \xi_k&=& \Delta \rho_{k-1} \xi_{k-1} +  \sum_{l\in K(n,t_1,t_k,k)\backslash \{K(n,t_1,t_{k-1},{k-1})\}} \Delta \rho_l \xi_l, \\
&=&\sum_{i=1}^{k} \sum_{l\in K(n,t_1,t_i,i)\backslash \{K(n,t_1,t_{i-1},{i-1})\}} \Delta \rho_l \xi_l.\\
\end{array}
\end{equation}

Above, by convention $K(n,t_1,t_{0},0)=\emptyset$. Using \eqref{gerg2} and \eqref{2222}, we conclude that 
\begin{equation} \label{f3}
\xi_b(t_m^+)= \prod_{i=1}^m (1-\psi^{i})\xi_b(t_1^-)+
\sum_{i=1}^{m}W_i \left( \sum_{k\in K(n,t_1,t_i,i)\backslash K(n,t_1,t_{i-1},i-1)} \Delta \rho_k \xi_k \right),
\end{equation}
with 
\begin{equation} \label{r1}
\begin{array}{l}
W_i:=\left( \frac{\psi^{m}}{\Delta \rho_m}+ \sum_{k=i}^{m-1} \frac{\psi^{k}}{\Delta \rho_{p_{k}}} \prod_{j=0}^{m-k-1} ((1-\psi^{m-j}) \right), \quad i=\{1,\cdots,m-1\},\\
W_m=\frac{\psi^{m}}{\Delta \rho_m}.
\end{array}
\end{equation}
 We notice that $W_{i}= W_{i+1}+\frac{\psi^{i}}{\Delta \rho_{i}} \prod_{j=0}^{m-i-1} (1-\psi^{m-j})$. We conclude that 
 \begin{equation} \label{cas1}
W_b(t_1^-,t_m^+)= \prod_{i=1}^m (1-\psi^{i}),
 \end{equation} and 
 \begin{equation} \label{cas2} W_k(t_1^-,t_m^+)=\left\{ \begin{array}{l}
W_m \quad \text{if} \, \, k\in K(n,t_1^-,t_m^+,m)\backslash K(n,t_1^-,t_{m-1}^+,m-1), \\
W_i \quad \text{if} \, \, k \in K(n,t_1^-,t_i,i)\backslash K(n,t_1^-,t_{i-1},i-1),  \quad i \in \{1,\cdots,m-1\},\\
0 \quad \text{otherwise},
\end{array}\right.
\end{equation}
with $W_m$ and $W_i$ defined in \eqref{r1}. From \eqref{lala}, $\vert W_b(t_1^-,t_m^+) \vert\leq \exp\left(\frac{3TV(\rho_0)}{\rho^*}\right)$. Using Lemma \ref{bababa},  we have $\vert \frac{\psi^{i}}{\Delta \rho_{p_{i}}} \vert \leq \frac{2}{\rho^*}$  for every $i\in \{1,\cdots,m\}$ and we conclude that
$$ 
\vert W_i\vert \leq  \vert \frac{\psi^{m}}{\Delta \rho_m}\vert + \sum_{k=i+2}^{m-1} \vert \frac{\psi^{k}}{\Delta \rho_{p_{k}}}\vert  \prod_{j=0}^{m-k-1} \vert (1-\psi^{m-j})\vert +\vert \frac{\psi^{i+1}}{\Delta \rho_{p_{i+1}}} \vert \prod_{j=0}^{m-i-2} \vert (1-\psi^{m-j})\vert +\frac{2}{\rho^*} \prod_{j=0}^{m-i-1} \vert (1-\psi^{m-j}) \vert. 
$$
Moreover, using Lemma \ref{bababa},
\begin{equation} \begin{array}{ll}
\vert \frac{\psi^{i+1}}{\Delta \rho_{p_{i+1}}} \prod_{j=0}^{m-i-2} (1-\psi^{m-j})+\frac{\psi^{i}}{\Delta \rho_{p_{i}}} \prod_{j=0}^{m-i-1} (1-\psi^{m-j}) &=  \prod_{j=0}^{m-i-2} (1-\psi^{m-j}) \left(\frac{\psi^{i+1}}{\Delta \rho_{p_{i+1}}} + (1-\psi^{i+1})\frac{\psi^{i}}{\Delta \rho_{p_{i}}} \right),\\
&\leq \prod_{j=0}^{m-i-2} \vert (1-\psi^{m-j})\vert \frac{2}{\rho^*}.
\end{array}
\end{equation}
Thus, 
$$  \vert W_i\vert   \leq  \vert \frac{\psi^{m}}{\Delta \rho_m}\vert + \sum_{k=i+3}^{m-1} \vert \frac{\psi^{k}}{\Delta \rho_{p_{k}}}\vert  \prod_{j=0}^{m-k-1} \vert (1-\psi^{m-j})\vert +\vert \frac{\psi^{i+2}}{\Delta \rho_{p_{i+2}}} \vert \prod_{j=0}^{m-i-3} \vert (1-\psi^{m-j})\vert +\frac{2}{\rho^*} \prod_{j=0}^{m-i-2} \vert (1-\psi^{m-j}) \vert.$$
By induction,  we conclude that 
$$\vert W_i \vert \leq \vert \frac{\psi^{m}}{\Delta \rho_m}\vert+ \vert (1-\psi^m) \vert \frac{2}{\rho^*} \leq  \frac{2}{\rho^*},$$
whence 
$$\max(W_b(t_1^-,t_2^+),W_k(t_1^-,t_2^+))\leq  \frac{2 \exp\left(\frac{3TV(\rho_0)}{\rho^*}\right)}{\rho^*}.$$
 \ \\
 \par  \textbf{General case:} we assume that there exist $m$ wave denoted by $1,\cdots,m$ interacting with the SV trajectory at $t=t_1, t=t_2, \cdots, t=t_m$ respectively as a wave-SV interaction modifying the speed of the SV  with $1\prec_n 2 \prec_n \cdots \prec_n \prec_n m$, that is to say $i=\{1,\cdots,m-1\}$, $i\in K(n,t_i,t_{i+1},i+1)$.  Let  $t_i$ be the first interation time after $t_{i-1}$ such that $i-1 \prec_n i$ for $i=\{2,\cdots,m\}$.
  Besides, for $i=\{1,\cdots,m-1\}$, we may have multiple  $n_i$ wave-SV interactions coming from the right on $(t_{i},t_{i+1})$.
 More precisely, for $j=\{1,\cdots,n_i\}$, we assume that the $p_{i,j}$-wave interacts  with the SV trajectory as a  wave-SV interactions at time $t_{p_{i,j}}$ such that $t_i<t_{p_{i,1}}<\cdots<t_{p_{i,n_i}}<t_{i+1}$ and $p_{i,j} \in K(n,t_{p_{i,j}},t_{i+1}^-,i+1)$ et $i \notin K(n,t_i,t_{i+1}^-,p_{i,j})$. We introduce the  following notation
$$K_{j}(n,t_1):=\left\{K(n,t_1,t_j,j) \backslash \left( \cup_{q=0}^{n_j-1} K(n,t_1,t_{j,n_j-q},p_{j,n_j-q} \cup K(n,t_1,j-1))\right)\right\}.
$$
$K_j(n,t_1)$ denotes the set of classical shocks at time $t_1$ whom the elements are the ancestors of  the $j^{\text{th}}$-wave with $j\in K(n, t_j)$ and they never interacts again with the SV trajectory  over $[t_1,t_j)$. We notice that $$K(n,t_1,t_m,p_m)=\left(\cup_{j=1}^m K_j(n,t_1)\right) \cup \left( \cup_{j=2}^m \cup_{q=0}^{n_{j}-q}K(n,t_1,t_{j,n_j-q},p_{j,n_j-q})\right).$$
In the sequel, we will denote  $C_i$ by $C_i:= (1-\psi^{i})$. Combining \eqref{f2} and \eqref{f3} with straightforward computations we have 
\begin{equation} \label{fff}
\left\{\begin{array}{l}
\xi_b(t_m^+)=W_b^m(t_1^-,t_m^+) \xi_b(t_1^-)+\sum_{j=1}^{m}W_j^m \sum_{k \in K_j(n,t_1)} \Delta_k \rho_k,\\
\quad \quad \quad+\sum_{j=2}^{m} \sum_{q=0}^{n_j-1} W_{j,q}^m \sum_{k \in K(n,t_1,t_{j,n_j-q},p_{j,n_j-q})} \Delta_k \rho_k,\\
\Delta \rho_j \xi_j=\sum_{k \in K(n,t_1,t_m,j)} \Delta_k \xi_k, \quad \text{for every} \, \, j\in K(n,t_m),  \\
\end{array}\right.
\end{equation}
with 
\begin{equation}\label{a1}
 W_b^m(t_1^-,t_m^+)=\prod_{i=0}^{m-2} \left( C_{m-i}  \prod_{l=1}^{n_{m-i}} C_{p_{n_{m-i},l}} \right) (1-\psi^{1}),
 \end{equation}
\begin{equation} \label{a2}
W_j^m=\left|\begin{array}{l} 
\frac{\psi^{m}}{\Delta \rho_m} \quad \text{if} \, \, j=m, \\
\frac{\psi^{m}}{\Delta \rho_m}+ \sum_{k=j}^{m-1} \frac{\psi^{k}}{\Delta \rho_{p_{k}}} \prod_{i=0}^{m-k-1} C_{m-i}  \prod_{i=0}^{m-k-1} \prod_{l=1}^{n_{m-i}}C_{p_{n_{m-i},l}}, \quad \text{otherwise.}\\
\end{array}  \right.
\end{equation}
and for $j=\{1,\cdots,m\}$, $q=\{0,\cdots,n_j-1\}$, 
\begin{equation}  \label{a3345} W_{j,q}^m= W_{j}+ \prod_{i=0}^{m-j-1} \left(C_{m-i} \prod_{l=1}^{n_{m-i}}C_{p_{n_{m-i},l}}\right)C_j\prod_{l=n_j-q+1}^{n_{j}}C_{n_{j},l} \frac{\psi^{p_{j,n_j-q}}}{\Delta \rho_{p_{j,n_j-q}}}.
\end{equation}
If $q=0$, by convention we require that $\prod_{l=n_j-q+1}^{n_{j}}C_{n_{j},l}:=1.$ From \eqref{fff}, we conclude that 
\begin{equation} \label{b111}
\left\{\begin{array}{l}
\xi_b(t_m^+)=W_b(t_1^-,t_m^+) \xi_b(t_1^-)+ \sum_{k \in K(n,t_1^-)} W_k(t_1^-,t_m^+) \Delta_k \rho_k,\\
\Delta \rho_j(t_m^+) \xi_j(t_m^+)=\sum_{k \in K(n,t_1^-,t_m^+,j)} \Delta_k \xi_k, \quad \text{for every} \, \, j\in K(n,t_m^+),  \\
\end{array}\right.
\end{equation}
with $W_b^m(t_1^-,t_m^+)$ defined in \eqref{a1} and 
 \begin{equation} \label{cas3}W_k(t_1^-,t_m^+)=\left\{ \begin{array}{l}
W_j^m \quad \text{if} \, \, k\in K_j(n,t_1), \, j=\{1,\cdots,m\}, \\
W_{j,q}^m \quad \text{if} \, \, k \in k \in K(n,t_1,t_{j,n_j-q},p_{j,n_j-q}), \, i \in \{1,\cdots,m-1\}, \, q=\{0,\cdots,n_j\},\\
0 \quad \text{otherwise}.
\end{array}\right.
\end{equation}
where $W_j^m$ and $W_{j,q}^m$ are defined in \eqref{a3345}.
From now on, we prove that $W_b^m(t_1^-,t_m^+)$ defined in \eqref{a1} and $W_k(t_1^-,t_m^+)$ defined in  \eqref{cas3} are bounded independently of $n$. Let  $j=\{1,\cdots,m\}$, $q=\{0,\cdots,n_j-1\}$ . From Lemma \ref{bababa},  $\vert \prod_{i=0}^{m-j-1} \prod_{l=1}^{n_{m-i}}C_{n_{m-i},l}\vert \leq (1+\frac{3}{2^{n+1}\rho^*})^{n_m+\cdots+n_{j+1}}$ and $\max( \frac{\psi^{j}}{\Delta \rho_{p_{j}}},\frac{\psi^{p_{j,n_j-q}}}{\Delta \rho_{p_{j,n_j-q}}}) \leq \frac{2}{\rho^*}$. Since $\sum_{i=2}^m n_i \leq TV(\rho_0)2^{n+1}$ we have $ \vert \prod_{i=0}^{m-j-1} \prod_{l=1}^{n_{m-i}}C_{n_{m-i},l} \vert \leq \exp\left(\frac{3TV(\rho_0)}{\rho^*}\right)$. Thus, 
$$
\frac{\max(W_{j-1}^m,W_{j,q}^m)}{\exp\left(\frac{3TV(\rho_0)}{\rho^*}\right)} \leq  \vert \frac{\psi^{m}}{\Delta \rho_m}\vert+ \sum_{k=j}^{m-1} \vert \frac{\psi^{k}}{\Delta \rho_{k}}\vert \prod_{i=0}^{m-k-1} \vert C_{m-i}\vert+\prod_{i=0}^{m-j} \vert C_{m-i}\vert \frac{2}{\rho^*}.
$$
Since
$$
\sum_{k=j}^{m-1} \vert \frac{\psi^{k}}{\Delta \rho_{k}}\vert \prod_{i=0}^{m-k-1} \vert C_{m-i}\vert+\prod_{i=0}^{m-j} \vert C_{m-i}\vert \frac{2}{\rho^*}=$$
$$\sum_{k=j+1}^{m-1} \vert \frac{\psi^{k}}{\Delta \rho_{k}}\vert \prod_{i=0}^{m-k-1} \vert C_{m-i}\vert+\prod_{i=0}^{m-j-1} \vert C_{m-i}\vert \left(\vert \frac{\psi^{j}}{\Delta \rho_{j}}\vert+\vert 1- \frac{\psi^{j}}{\Delta \rho_{j}}\vert \frac{2}{\rho^*} \right).$$
From Lemma \ref{bababa} and Lemma \ref{pp}, 
$$\vert \frac{\psi^{j}}{\Delta \rho_[{j}}\vert+\vert 1- \frac{\psi^{j}}{\Delta \rho_{j}}\vert \frac{2}{\rho^*} \leq \frac{2}{\rho^*}.$$
We conclude that 
$$
\frac{\max(W_{j-1}^m,W_{j,q}^m)}{\exp(\frac{2^5}{\rho^*})} \leq  \vert \frac{\psi^{m}}{\Delta \rho_m}\vert+ \sum_{k=j+1}^{m-1} \vert \frac{\psi^{k}}{\Delta \rho_{k}}\vert \prod_{i=0}^{m-k-1} \vert C_{m-i}\vert+\prod_{i=0}^{m-j-1} \vert C_{m-i}\vert \frac{2}{\rho^*}.
$$
By induction, we deduce that 
$\frac{\max(W_{j-1}^m,W_{j,q}^m)}{\exp\left(\frac{3TV(\rho_0)}{\rho^*}\right)} \leq \frac{2}{\rho^*},$
whence, for every $k\in K(n,t_1^-)$,
$$\max(W_b^m(t_1^-,t_m^+),W_k(t_1^-,t_m^+)) \leq \frac{2\exp\left(\frac{3TV(\rho_0)}{\rho^*}\right)}{\rho^*},$$
with $W_b^m(t_1^-,t_m^+)$ defined in  \eqref{a1} and $W_k(t_1^-,t_m^+)$ defined  \eqref{cas3}.
We conclude the proof of Lemma \ref{coco} noticing for every $k\in K(n,t_1)$ that there exist $m \in N^*$, $j \in \{1,\cdots,m \}$, $n_j \in \N^*$ and $q\in \{0,\cdots,n_j-1\}$ such that 
$$ W_{k}(t_1,t_2)=\left\{ \begin{array}{l}
W_j^m, \\
W_{j,q}^m,\\
0,
\end{array}\right.$$
where $W_{k}(t_1,t_2)$ is defined in \eqref{ffffefe}, $W_j^m$ is defined in \eqref{a2} and $W_{j,q}^m$ is defined in \eqref{a3345}.


\footnotesize
\bibliographystyle{plain}
\bibliography{biblio}
\end{document}